\documentclass[11pt,reqno,tbtags]{amsart}

\usepackage{amsthm}
\usepackage{amssymb}
\usepackage{amsfonts, mathrsfs}
\usepackage{amsmath}
\usepackage[numbers, sort&compress]{natbib}
\usepackage{graphicx}
\usepackage{vmargin, enumerate}
\usepackage{setspace}
\usepackage{paralist}
\usepackage{hyperref}
\usepackage{color,xcolor}

\newcommand{\Z}{\mathbb{Z}}
\newcommand{\N}{{\mathbb N}}

\newcommand{\E}[1]{{\mathbf E}\left[#1\right]}
\newcommand{\e}{{\mathbf E}}

\newcommand{\p}[1]{{\mathbf P}\left\{#1\right\}}

\newcommand{\I}[1]{{\mathbf 1}_{[#1]}}
\newcommand{\set}[1]{\left\{ #1 \right\}}

\newcommand{\phat}[1]{\ensuremath{\hat{\mathbf P}}\left\{#1\right\}}
\newcommand{\Ehat}[1]{\ensuremath{\hat{\mathbf E}}\left[#1\right]}

 \newcommand{\bag}{\begin{align}}
\newcommand{\bags}{\begin{align*}}
\newcommand{\eag}{\end{align*}}
\newcommand{\eags}{\end{align*}}

\newtheorem{thm}{Theorem}[section]
\newtheorem{lem}[thm]{Lemma}
\newtheorem{prop}[thm]{Proposition}
\newtheorem{cor}[thm]{Corollary}

\newcommand\cC{\mathcal C}

\newcommand\cF{\mathcal F}
\newcommand\cG{\mathcal G}

\newcommand\cM{\mathcal M}

\newcommand\cT{{\mathcal T}}

\newcommand{\refT}[1]{Theorem~\ref{#1}}

\newcommand{\refL}[1]{Lemma~\ref{#1}}

\newcommand{\pran}[1]{\left(#1\right)}

\hypersetup{
    bookmarks=true,         
    unicode=false,          
    pdftoolbar=true,        
    pdfmenubar=true,        
    pdffitwindow=true,      
    pdftitle={My title},    
    pdfauthor={Author},     
    pdfsubject={Subject},   
    pdfnewwindow=true,      
    pdfkeywords={keywords}, 
    colorlinks=true,       
    linkcolor=blue,          
    citecolor=blue,        
    filecolor=blue,      
    urlcolor=blue           
}

\reversemarginpar
\definecolor{clou}{rgb}{0.5,0.125,0.3125}

\newcommand{\eps}{\epsilon}

\newcommand\urladdrx[1]{{\urladdr{\def~{{\tiny$\sim$}}#1}}}

\begingroup
  \divide\count255 by 60
  \multiply\count255 by -60
  \advance\count255 by \time
  \ifnum \count255 < 10 \xdef\oclock{\the\count1:0\the\count255}
  \else\xdef\oclock{\the\count1:\the\count255}\fi
\endgroup

\DeclareMathOperator{\PGW}{PGW}
\DeclareMathOperator{\pgw}{PGW}
\newcommand{\eqdist}{\ensuremath{\stackrel{\mathrm{d}}{=}}}
\newcommand{\convdist}{\ensuremath{\stackrel{\mathrm{d}}{\to}}}

\definecolor{mygrey}{gray}{0.97}
\usepackage{lettrine}
\newcommand{\floor}[1]{\ensuremath{\lfloor #1 \rfloor}}
\newcommand{\ceil}[1]{\ensuremath{\lceil #1 \rceil}}
\newcommand{\diam}{\ensuremath{\mathrm{diam}}}
\newcommand{\leas}{\ensuremath{\le_{\mathrm{a.s.}}}}
\newcommand{\geas}{\ensuremath{\ge_{\mathrm{a.s.}}}}
\newcounter{CC}
\newcommand{\CC}{\stepcounter{CC}\CCx} 
\newcommand{\CCx}{C_{\arabic{CC}}}     
\newcommand{\CCdef}[1]{\xdef#1{\CCx}}     
\newcounter{cc}
\newcommand{\cc}{\stepcounter{cc}\ccx} 
\newcommand{\ccx}{c_{\arabic{cc}}}     
\newcommand{\ccdef}[1]{\xdef#1{\ccx}}     
\newcommand{\convas}{\ensuremath{\stackrel{\mathrm{a.s.}}{\rightarrow}}}
\renewcommand{\d}{\ensuremath{\mathrm{d}}}

\newcommand{\pst}{\ensuremath{\preceq_{\mathrm{st}}}}
\newcommand{\sust}{\ensuremath{\succeq_ {\mathrm{st}}}}
\newcommand{\zpu}{\ensuremath{z\mbox{-{\it Prim}}(u)}}

\newcommand{\Cexp}[2]{\mathbf{E}\set{\left. #1 \; \right| \; #2}}

\numberwithin{equation}{section}

\usepackage[refpage,noprefix]{nomencl}

\makenomenclature

\newcommand{\rG}{\ensuremath{\mathrm{G}}}
\newcommand{\rT}{\ensuremath{\mathrm{T}}}
\newcommand{\rK}{\ensuremath{\mathrm{K}}}
\newcommand{\rU}{\ensuremath{\mathrm{U}}}
\newcommand{\rM}{\ensuremath{\mathrm{M}}}
\newcommand{\rY}{\ensuremath{\mathrm{Y}}}

\begin{document}

\title[The local weak limit of the MST of $K_n$]{The local weak limit of the minimum spanning tree of the complete graph}

\author{L. Addario-Berry}
\address{Department of Mathematics and Statistics, McGill University, 805 Sherbrooke Street West, 
		Montr\'eal, Qu\'ebec, H3A 2K6, Canada}
\email{louigi@math.mcgill.ca}
\date{January 8, 2013} 
\urladdrx{http://www.math.mcgill.ca/~louigi/}


\begin{abstract} 
Assign i.i.d.\ standard exponential edge weights to the edges of the complete graph $K_n$, and let $M_n$ be the resulting minimum spanning tree. We show that $M_n$ converges in the local weak sense (also called Aldous-Steele or Benjamini-Schramm convergence), to a random infinite tree $M$. The tree $M$ may be viewed as the component containing the root in the wired minimum spanning forest of the Poisson-weighted infinite tree (PWIT). We describe a Markov process construction of $M$ starting from the invasion percolation cluster on the PWIT. We then show that $M$ has cubic volume growth, up to lower order fluctuations for which we provide explicit bounds. Our volume growth estimates confirm recent predictions from the physics literature \cite{jackson2010theory}, and contrast with the behaviour of invasion percolation on the PWIT \cite{addario12prim} and on regular trees \cite{angel08invasion}, which exhibit quadratic volume growth. 
\end{abstract} 

\maketitle

\vspace{-0.5cm}
\tableofcontents

\section{{\bf Introduction}}\label{sec:intro} 

\lettrine[lines=2]{A}{} very recent preprint, written by the author of this paper and Broutin, Goldschmidt and Miermont \cite{us2013}, identified the Gromov-Hausdorff-Prokhorov scaling limit of the minimum spanning tree of the complete graph and described some of its basic properties. The current paper complements the work of \cite{us2013} by identifying the {\em un-rescaled} or {\em local weak} limit of the minimum spanning tree of the complete graph. Both the work of \cite{us2013} and the current work open avenues for further research, into developing a more detailed understanding of the properties of the limiting objects. We highlight some specific questions of interest later in the introduction; for the moment, we jump right to the definitions required for a precise statement of our results. 

\subsection*{Rooted weighted graphs}
\nomenclature[RWG]{RWG}{Rooted weighted graph}
A {\em rooted weighted graph} (RWG) is a triple $\rG=(G,\rho,w)$, where $G=(V(G),E(G))$ is a connected simple graph with countable vertex degrees, $\rho \in V(G)$ is the {\em root vertex}, and $w:E(G) \to [0,\infty)$ assigns weights to the edges of $G$. By convention, we view unweighted rooted graphs with as RWGs, by letting all edges have weight $1$. 

If $\rG=(G,\rho,w)$ is an RWG and $G'$ is a connected subgraph of $G$ containing $\rho$ (i.e., $G'$ is a graph with $\rho \in V(G') \subset V(G)$ and $E(G') \subset E(G)$), then $(G',\rho,w|_{E(G')})$ is another RWG, and is a {\em sub-RWG} of $\rG$. We sometimes write $(G',\rho,w)$ instead of $(G',\rho,w|_{E(G')})$ for succinctness, when doing so is unlikely to cause confusion. Given $S \subset V(G)$ with $\rho \in S$, we write $\rG[S]$ for the sub-RWG of $\rG$ induced by $S$. 

For an RWG $\rG=(G,\rho,w)$, the {\em graph distance} $d_{\rG}:V(G) \times V(G) \to [0,\infty]$ is given by 
\nomenclature[D]{$d_{\rG}$}{Unweighted graph distance in $\rG$}
\[
d_{\rG}(u,v) = \inf\left\{ |E(P)|~:~ P~\mbox{a path from $u$ to $v$ in $G$}\right\}\, ,
\]
for $u,v \in V(G)$. (Elsewhere we let $\inf \emptyset = \infty$, but note that here the infimum is non-empty as $G$ is connected.) We similarly define the {\em weighted} graph distance by 
\nomenclature[D]{$d'_{\rG}$}{Weighted graph distance in $\rG$}
\[
d'_{\rG}(u,v) = \inf\left\{ \sum_{e \in E(P)} w(e)~:~ P~\mbox{a path from $u$ to $v$ in $G$}\right\}\, ,
\]
For $v \in V(G)$ and $x > 0$, we write 
$B_{\rG}(v,x) = \{u \in V(G):d_{\rG}(u,v) \le x\}$ and 
$B'_{\rG}(v,x) = \{u \in V(G):d'_{\rG}(u,v) \le x\}$. 
\nomenclature[BGvx]{$B_{\rG}(v,x)$}{Ball in $\rG$ for the distance $d_{\rG}$.}
\nomenclature[Bgvx]{$B'_{\rG}(v,x)$}{Ball in $\rG$ for the distance $d_{\rG}'$.}
We say $\rG$ is {\em locally finite} if $|B'_{\rG}(r,x)| < \infty$ for all $x\ge 0$. 

\subsection*{Prim's algorithm/invasion percolation}
We briefly recall the definition of {\em Prim's algorithm} (also called {\em invasion percolation}) on an RWG. Let $\rG=(G,\rho,w)$ be a locally finite RWG with all edge weights distinct. 

\medskip
\fcolorbox{black}{mygrey}{%
\begin{minipage}{0.95\textwidth}%
\noindent {\it Prim's algorithm on $\rG$.}\\ 
Let $E_1=E_1(\rG)=\emptyset$ and let $v_1=v_1(\rG)=\rho$. \\
For $1 \le i < |V(\rG)|$:\\
\hspace*{0.5cm}$\star$ let $e_i=e_i(\rG) \in E(G)$ minimize 
$\{w(e): e=uv, u \in \{v_1,\ldots,v_i\},v \not\in\{v_1,\ldots,v_i\}\}$; \\
\hspace*{0.5cm}$\star$ let $v_{i+1}=v_{i+1}(\rG) = v$ and let $E_{i+1}=E_{i+1}(\rG) = E_i \cup \{e_{i+1}\}$. 
\end{minipage}
}
\nomenclature[Vig]{$v_i(\rG)$}{The $i$'th vertex added by Prim's algorithm on $\rG$.}
\nomenclature[Eig]{$e_i(\rG)$}{The $i$'th edge added by Prim's algorithm on $\rG$.}

\medskip
If $\rG$ is finite then, writing $n=|V(G)|$, the graph $(V(G),E_n)$ is the unique 
tree $T=(V(T),E(T))$ with $V(T)=V(G)$ and $E(T) \subset E(G)$ minimizing $\sum_{e \in E(T)} w(e)$; in other words, it is the {\em minimum weight spanning tree} (MST) of $\rG$. 
If $\rG$ is infinite, however, $\{v_i, i \ge 1\}$ may be a strict subset of $V(G)$, in which case the tree constructed by Prim's algorithm is not a {\em spanning} tree; we call it the {\em invasion percolation cluster} of $\rG$. 

There is a large corpus on minimum spanning trees/invasion percolation clusters of random graphs, with contributions from both the combinatorics and probability communities (we refer the interested reader to \cite{us2013} for a detailed bibliography). In many settings, minimum spanning trees and forests turn out to be intimately linked to an associated percolation process. This can serve as both a motivation and a warning (to justify the latter epithet, we record Newman's observation \cite{newman97topics} that, weighting the edges of $\Z^d$ by iid Uniform$[0,1]$ edge weights, the invasion percolation cluster has asymptotically zero density if and only if the percolation probability $\theta_{p_c}(\Z^d)=0$). 

\subsection*{Local weak limits}

Our current aim is to study the local structure of the MST of the complete graph. More precisely, for $n \ge 1$ we write $K_n$ for the graph with vertices $[n]=\{1,\ldots,n\}$ and an edge between each pair of vertices. Let $W_n=\{W_n(e):e \in E(K_n)\}$ be iid Exponential$(n-1)$ edge weights, and write $\rK_n = (K_n,1,W_n)$. 
\nomenclature[Kn]{$\mathrm{K}_n$}{The complete graph on $[n]$ with root $1$ and with Exponential$(n-1)$ edge weights.}
Then let 
$M_n = ([n],E_n(\rK_n))$ be the minimum spanning tree of $\rK_n$, and let 
$\rM_n = (M_n,1,W_n)$ be the RWG formed from $M_n$ by rooting at the vertex~$1$. 
\nomenclature[Mn]{$\rM_n$}{The minimum spanning tree of $\rK_n$; $\rM_n=(M_n,1,W_n)$.}
The primary contributions of this paper are to identify the local weak limit of $\rM_n$ and to prove volume growth bounds for the limiting RWG. 

\medskip 
In order to precisely state our results, we first recall the notion of local weak convergence. For unweighted graphs, such convergence was introduced by 
\citet{bs}. The formulation we use here is essentially that of \citet{aldous2004omp}. 
For $\eps \ge 0$, we say two RWGs $\rG=(G,\rho,w)$ and $\rG'=(G',\rho',w')$ are {\em $\eps$-isomorphic}, and write $\rG\simeq_{\eps} \rG'$, if there exists a bijection $\phi:V(G) \to V(G')$ that induces a graph isomorphism of $G$ and $G'$, such that $\phi(\rho)=\rho'$ and such that for all edges $uv \in e(G)$, $|w(uv)- w'(\phi(u)\phi(v))| \le \eps$. If $\eps=0$ we say that $\rG$ and $\rG'$ are isomorphic and write $\rG\simeq \rG'$. 

We may define a pseudometric $d_0$ on the set of locally finite RWGs as follows. 
Given RWGs $\rG_1=(G_1,\rho_1,w_1)$ and $\rG_2=(G_2,\rho_2,w_2)$, we let
\[
d_0(\rG_1,\rG_2) = \frac{1}{2^k}, 
\]
where
\[
n = \sup\left\{k \in \N: \exists \delta \in [-k^{-1},k^{-1}]~,~\rG_1'(k)\simeq_{k^{-1}} \rG_2'(k+\delta)\right\}\, .
\]
Note that $d_0(\rG_1,\rG_2)=0$ precisely if $\rG\simeq \rG'$. 
Now let $d_{\mathrm{LWC}}$ be the push-forward of $d_0$ to the set $\cG^*$ of isomorphism-equivalence classes of locally finite RWGs.\footnote{When referring to an equivalence class $[(G,\rho,w)] \in \cG^*$, we will typically instead write $(G,\rho,w)$; this slight 
abuse of notation should never cause confusion.} 
It is straightforward to verify that $(\cG^*,d_{\mathrm{LWC}})$ forms a complete, separable metric space, and we refer to convergence in distribution in this metric space as local weak convergence, or sometimes simply as weak convergence. 

\subsection{Statement of results} 
As is the case for many random combinatorial optimization problems on the complete graph, the local weak limit of $\rM_n$ is 
naturally described in terms of the {\em Poisson-weighted infinite tree} (PWIT) of \citet{aldous2004omp}. The PWIT is the following random RWG. 
Let $U$ be the Ulam-Harris tree; 
\nomenclature[U]{$U$}{The Ulam-Harris Tree}
\nomenclature[Uow]{$(U,\emptyset,W)$}{The Poisson-weighted infinite tree}
this is the tree with vertex set $\bigcup_{n \ge 0} \N^n$ (write $\N^0=\{\emptyset\}$), and 
for each $k \ge 1$ and each vertex $v=(n_1,\ldots,n_k) \in \N^k$, an edge between $v$ and its parent $(n_1,\ldots,n_{k-1})$. (If $k=1$ then 
the parent of $v$ is the root vertex $\emptyset$.) Independently for each $v = (n_1,\ldots,n_k) \in V$, let $(w_i,i \ge 1)$ be the atoms of a
homogenous rate one Poisson process on $[0,\infty)$, and for each $i \ge 1$ give the edge from $v$ to its child $v'=(n_1,\ldots,n_k,i)$ the weight $W(\{v,v'\})=w_i$. Writing $W=\{W(e),e \in E\}$, where $E$ is the edge set of $U$, the Poisson-weighted infinite tree is (a random RWG with the distribution of) the triple $\rU=(U,\emptyset,W)$. 

For any vertex $u \in V(U)$, write $T^{(u)}$ for the invasion percolation cluster of $\rU^{(u)}=(U,u,W)$. 
\nomenclature[Tu]{$T^{(u)}$}{The invasion percolation cluster of $(U,u,W)$; $T=T^{(\emptyset)}$.}
\nomenclature[M]{$\rM$}{Component of the wired MSF on $\rU$ containing the root $\emptyset$; $rM=(M,\emptyset,W)$.}
The tree $T^{(\emptyset)}$ will play a particularly important role, so we write $T=T^{(\emptyset)}$ and 
write $\rT=(T,\emptyset,W)$. 
Temporarily let $M_0$ be the union of all the trees $(T^{(u)},u \in V(U))$; 
$M_0$ should be thought of as the minimum spanning {\em forest} of $\rU$ with wired boundary conditions. Next, let $M$ be the connected component of $M_0$ containing the root $\emptyset$. In this paper we establish that $\rM=(M,\emptyset,W)$ is the local weak limit of the minimum spanning tree $\rM_n=(M_n,1,W_n)$. 
\begin{thm}\label{thm:main}
As $n \to \infty$, we have $(M_n,1,W_n)\convdist(M,\emptyset,W)$, in the local weak sense. 
\end{thm}
Theorem~\ref{thm:main} in particular extends a result of Aldous (\cite{aldous90randomtree}, Theorem 1). There is a.s\ a unique edge $e \in E(M)$ incident to $\emptyset$ whose removal separates $\emptyset$ from $\infty$ 
(this follows from Corollary~\ref{cor:nvdist}, below, which states that $M$ is one-ended). 
Write $F$ for the component containing $\emptyset$ after $e$ is removed. Correspondingly, let $e_n$ be the edge of $M_n$ whose removal minimizes the size of the resulting component containing $1$ (with ties broken by choosing the component containing the smallest label, say), and write $F_n$ for this component. Aldous proved that $(F_n,1,W_n)$ converges in distribution in the local weak sense, to an almost surely finite limit, which in our setting has the distribution of $(F,\emptyset,W)$. The convergence of $(F_n,1)$ to $(F,\emptyset)$ follows immediately from Theorem~\ref{thm:main}. 

The proof of Theorem~\ref{thm:main} will also provide an explicit characterization of the limit, which is straightforward enough to yield an exact, though somewhat complicated, description of the distribution of the degree of the root vertex $\emptyset$ in $\rM$. This description is provided in Section~\ref{sec:consequences}. 

Our second main result is to show that the limit $\rM$ has roughly cubic volume growth. 
\begin{thm}\label{thm:volume}
There is $C > 0$ such that almost surely 
\[
\limsup_{r \to \infty} \frac{|B_{\rM}(\emptyset,r)|}{r^3 e^{C\log^{1/2} r}} < \infty\, ,
\]
and almost surely
\[
\liminf_{r \to \infty} \frac{|B_{\rM}(\emptyset,r)|}{r^3/ \log^{22}r} > 0\, .
\]
\end{thm}
This volume growth agrees with recent predictions from the physics literature \cite{jackson2010theory}, and, as shown by Theorem~\ref{thm:tvolume}, stands in contrast to the volume growth of\, $\rT=(T,\emptyset,W)$, the invasion percolation cluster of $(U,\emptyset,W)$. 
\begin{thm}\label{thm:tvolume} 
It is almost surely the case that 
\[
\limsup_{r \to \infty} \frac{|B_{\rT}(\emptyset,r)|}{r^2 \log \log r} < \infty\, ,
\]
and that for any $\eps > 0$, 
\[
\liminf_{r \to \infty} \frac{|B_{\rT}(\emptyset,r)|}{r^2/\log^{8+\eps}} = \infty\, .
\]
\end{thm}
In Theorems~\ref{thm:volume} and~\ref{thm:tvolume}, we have stated our bounds in terms of the unweighted graph distance. However, using the explicit characterizations of $\rM$ and $\rT$ given in Section~\ref{sec:distribution_info}, it is a straightforward technical exercise to establish the same bounds for the growth of $B'_{\rM}(\emptyset,r)$ and $B'_{\rT}(\emptyset,r)$. 

A similar dichotomy of dimensionality, paralleling the difference in volume growth between $\rT$ and $\rM$, appears when studying the Gromov-Hausdoff-Prokhorov  scaling limit of the minimum spanning tree; see \cite{us2013} for details. Broadly speaking, 
the tree $\rT$, which is a subtree of $\rM$, determines the global metric structure of $\rM$, in the sense that for ``typical'' nodes $u,v$ of $\rM$, the distance $d_{\rM}(u,v)$ is of about the same order as $d_{\rM}(\hat{u},\hat{v})$, where $\hat{u}$ and $\hat{v}$ are the nearest nodes in $\rT$ to $u$ and $v$, respectively. However, the bulk of the mass in $\rM$ lies outside of the subtree $\rT$. Theorems~\ref{thm:volume} and~\ref{thm:tvolume} may be viewed as a step towards formalizing this picture. 

We expect that $|B_{\rM}(\emptyset,r)|/r^3$ forms a tight sequence but that almost surely 
\[
\liminf_{r \to \infty} \frac{|B_{\rM}(\emptyset,r)|}{r^3}=0, \quad \limsup_{r \to \infty} \frac{|B_{\rM}(\emptyset,r)|}{r^3}=\infty\, .
\]
A proof of any of these predictions would be interesting. It would also be be interesting to pin down precisely the almost sure fluctuations of 
$|B_{\rM}(\emptyset,r)|/r^3$ (assuming this is the right renormalization), or even of $|B_{\rT}(\emptyset,r)|/r^2$. The work of \cite{duquesne06hausdorff} on the exact Hausdorff measure function for the Brownian CRT heuristically suggests that the almost sure fluctuations of $|B_{\rT}(\emptyset,r)|/r^2$ may be of order $\log\log r$, but the connection between the two settings is rather tenuous. 

Theorem~\ref{thm:tvolume} should be compared with \cite{angel08invasion}, Theorem 1.6. 
The latter shows that, writing $\rY(d)$ for the local weak limit of invasion percolation on an infinite $d$-ary tree, 
$|B_{\rY(d)}(\emptyset,r)|/(dr^2)$ converges in distribution as $r \to \infty$ (and explicitly describes the Laplace transform of the limit). Theorem 1.9 of \cite{angel08invasion} additionally shows that 
$\limsup_{r \to \infty} \E{r^2/|B_{\rY(d)}(\emptyset,r)|} < \infty$. However, we do not see how to deduce our Theorem~\ref{thm:tvolume} from the results of \cite{angel08invasion}. 

We finish the description of our contributions with a suggestion for future work: it would be quite interesting to understand the spectral and diffusive properties of $M$. Results of Barlow et.\ al.\ \cite{barlow2008random} and of Kumagai and Mizumi \cite{kumagai2008heat} suggest that for simple random walk $(x_n,n \ge 1)$ on $M$, started from the root $\emptyset$, we should have $\p{x_n=\emptyset} = n^{-3/4+o(1)}$ and $d_{\rM}(\emptyset,x_n) = n^{1/4+o_p(1)}$. However, the volume growth upper bound of Theorem~\ref{thm:volume} is not sharp enough to allow such results to be applied ``out of the box'', and neither do the requisite resistance bounds follow immediately from the current work. 

\subsection{Proofs of the main results (a brief sketch)}\label{sec:briefsketch}
The upper bound in Theorem~\ref{thm:tvolume} is straightforward. An earlier paper \cite{addario12prim} showed that $\rT$ is stochastically dominated by a Poisson$(1)$ Galton-Watson tree conditioned to be infinite. This fact, together with existing volume growth estimates for the incipient infinite cluster on trees \cite{barlow06iic}, yields the upper bound. For the lower bound in Theorem~\ref{thm:tvolume}, we analyze an explicit description of the distribution of $\rT$ from \cite{addario12prim}. This description, reviewed in Section~\ref{sec:wwlprim}, below, states that $\rT$ is comprised of a sequence of Poisson$(1)$ Galton-Watson trees, always of random size but conditioned to be increasingly large as the sequence goes on, and glued together along a backbone. Roughly speaking (ignoring logarithmic corrections), this allows us to find a subtree of $T$ distributed as a conditioned Poisson Galton-Watson tree with around $r^2$ vertices, and containing a backbone node of distance about $r$ from the root. Such a (sub)tree also has diameter roughly $r$, and this yields the lower bound. 

To prove Theorems~\ref{thm:main} and \ref{thm:volume}, we introduce, and then study, a two-step construction of the minimum spanning tree $\rM_n$. The construction is more elegant in the $n \to \infty$ limit, so we now outline the construction for $\rM=(M,\emptyset,W)$. 

We associate to each vertex $v \in V(M)$ an {\em arrival time} $a(v)$. Recall that $\rT=(T,\emptyset,W)$ is the invasion percolation cluster of $(U,\emptyset,W)$; its distribution will be explicitly described in Section~\ref{sec:wwlprim}. Given a vertex $v \in V(M)$, if $v \in V(T)$ then set $a(v)=1$. If $v \not\in V(T)$ then write $P_v$ for the shortest path in $U$ from $v$ to $T$ (this path lies within $M$), and let $a(v) = \max\{W(e), e \in E(P_v)\}$. 
\nomenclature[Av]{$a(v)$}{For $v \in V(T)$, $a(v)=1$; for $v \not\in V(T)$, $a(v)$ is the weight of the largest weight edge on the path from $v$ to $T$.}
Next, for $\lambda \ge 1$, let $M(\lambda)$ be the subtree of $M$ with vertex set $V(M(\lambda))=\{v \in V(M): a(v) \le \lambda\}$. 
\nomenclature[Mlambda]{$M(\lambda)$}{Subtree of $M$ with vertex set $\{v \in V(M): a(v) \le \lambda\}$.}
It will turn out that, almost surely, $1 < a(v) < \infty$ for all $v \not\in V(T)$, so $M(1)=T$ and $M=\lim_{\lambda \to \infty} M(\lambda)$. We view $M$ as being built in two steps; first, the tree $T$ is constructed (via Prim's algorithm, say). Second, the remaining vertices of $M$ are dynamically exposed by letting $\lambda$ grow from $0$ to $\infty$. 

For $\lambda \ge 1$ let $\rM(\lambda)=(M(\lambda),\emptyset,W)$. 
It turns out that, conditional on $\rT=\rM(1)$, the stochastic process $(\rM(\lambda),0 \le \lambda \le \infty)$ is Markovian (this is not a surprise) and its transition kernel has an explicit and pleasing form. This picture builds upon on the description of $\rT$ from given in Section~\ref{sec:wwlprim}, so we postpone the details until Section~\ref{sec:pwag}. 

We now sketch the dynamics of $(M(\lambda),\lambda \ge 1)$. 
In brief, at each time $\lambda > 1$, there is some set of {\em inactive} vertices $I_{\lambda}$ in $M(\lambda)$ - these vertices form a subtree of $T$ containing $\emptyset$. The process $(I_{\lambda}, \lambda \ge 1)$ is decreasing in $\lambda$, it satisfies $\lim_{\lambda \downarrow 1} I_{\lambda} = V(T)$, and there is some almost surely finite time $\lambda^*$ at which $|I_{\lambda^*}|=0$. 

Once a vertex $v$ is active, subcritical Poisson Galton-Watson trees begin to attach themselves to $v$ according to a particular inhomogeneous rate function which decays exponentially quickly in $\lambda$; this process happens independently for each active vertex, and is responsible for the growth of $M(\lambda)$ as $\lambda$ increases. When a tree attaches to $v$, all its vertices immediately become active. We dub this a {\em Poisson Galton-Watson aggregation process}.

By combining the volume growth upper bound for $\rT$ from Theorem~\ref{thm:tvolume} with a direct analysis of the Poisson Galton-Watson aggregation process, we are able to prove the volume growth upper bound from Theorem~\ref{thm:volume}. It seems likely that the lower bound should also be provable ``in the limit'', but I have been unable to 
achieve such an argument. For the lower bound, I instead resort to a somewhat involved second moment calculation, based on an analysis 
of modified versions of Prim's algorithm, started from three distinct vertices (it turns out to be necessary to consider two vertices aside from the root, which is why three distinct vertices are needed for a second moment argument). 

Finally, the proof of Theorem~\ref{thm:main} proceeds by analyzing a hybrid MST algorithm. 
Given a small input parameter $\eps > 0$, the we first Prim's algorithm for a random number of steps, halting at a (stopping) time $\tau$ that is $\Theta(\eps n)$ with high probability\footnote{In fact, $\tau$ is of order $(2+o_{\eps \downarrow 0}(1))\eps n$, but this information is unimportant to the informal description.}. This is the finite-$n$ analogue of the limiting ``step one'' and builds a random RWG $\rM_n(\tau)$ 
whose distribution is close to that of $\rT$ for $\eps$ small. 
Starting from $\rM_n(\tau)$, we then run Kruskal's algorithm to complete the construction of $\rM_n$. This second phase corresponds to the limiting ``step two''. Its relatively straightforward analysis in Section~\ref{sec:mainproof} will yield Theorem~\ref{thm:main}.

\subsection{Outline}
Before turning to proofs, we briefly outline the remainder of the paper. 
In Section~\ref{sec:distribution_info}, we describe a Markovian construction of the invasion percolation cluster $\rT=(T,\emptyset,W)$ that was established in an earlier work \cite{addario12prim} and that will be exploited several times in the current paper. We additionally we describe a continuous-time process that builds $\rM$ from $\rT$; this process is at the heart of the proof of Theorem~\ref{thm:main} and is also used in proving Theorem~\ref{thm:volume}. Finally, in Section~\ref{sec:consequences} we note some distributional identities that may be obtained from our results. 

In Section~\ref{sec:futuremax} we analyze the early stages (the first $o(n)$ steps) of Prim's algorithm on $\rK_n$, and show that the RWG thereby constructed has $\rT$ as its local weak limit. This fact allows us to define a finite-$n$ growth procedure analogous to $(\rM(\lambda),\lambda \ge 1)$, which is used in proving Theorem~\ref{thm:main}. 

In Section~\ref{sec:formax} we establish some useful properties of the decay of weights along the unique infinite path in $\rT$. These properties come into play
in Section~\ref{sec:volgor}, in which we prove Theorem~\ref{thm:tvolume}, and also in proving the lower bound of Theorem~\ref{thm:volume}. 

In Section~\ref{sec:mainproof}, we prove Theorem~\ref{thm:main}. As discussed in the proof sketch above, this proof is based on the analysis of a hybrid MST algorithm. 

In Section~\ref{sec:volgorm} we prove Theorem~\ref{thm:volume}. The upper bound is based on a direct analysis of the process $\rM(\lambda)$, and ultimately reduces to bounding the value of a certain iterated integral. The lower bound, based on a second moment argument which is the most technical element of this work, occupies the majority of Section~\ref{sec:volgorm}. 

Finally, we alert the reader that just before the bibliography, we have provided a list of some of the notation used in the paper, with brief reminders of the definitions and with page references. 
 
\subsection{Acknowledgements}
I started this work in July 2012, after Itai Benjamini and Shankar Bhamidi separately asked me what was known about the subject. I thank both of them for the encouragement to pursue this line of inquiry. Throughout my work on this project I was supported by an NSERC Discovery Grant and by an FQRNT Nouveau Chercheur grant.

\section{{ {\bf The distributions of\, $\rT$ and of\, $\rM$} }} \label{sec:distribution_info}
\subsection{An explicit construction of\, $\rT$. }\label{sec:wwlprim}
The following description of the distribution of $\rT$\, is given by Theorem~27 of \cite{addario12prim}.
Say that an edge $e \in E(T)$ is a {\em forward maximal edge} if the removal of $e$ separates $\emptyset$ from infinity and if, for any other edge $e' \in E(T)$, if the path from $e'$ to $\emptyset$ contains $e$ then $W(e') < W(e)$. Write $S_1=\emptyset$, and list the forward maximal edges of $T$ in increasing order of distance from $\emptyset$ as $(\{R_i,S_{i+1}\},i \ge 1)$. The removal of all forward maximal edges separates $T$ into an infinite sequence of random trees $(P_i,i \ge 1)$, where for each $i \ge 1$, $R_i$ and $S_i$ are vertices of $P_i$. 
\nomenclature[Pi]{$(P_i,i \ge 1)$}{The components of $T$ when forward maximal edges are removed.}

For each $i \ge 1$, let $Z_i=|V(P_i)|$ and let $X_i=W(\{R_i,S_{i+1}\})$. The sequence $((X_i,Z_i),i \ge 1)$ turns out to be a Markov process, whose distribution we now explain. 

For $\lambda \ge 0$, write $\pgw(\lambda)$ to denote a Galton-Watson tree with offspring distribution $\mathrm{Poisson}(\lambda)$. 
\nomenclature[PGW]{PGW$(\lambda)$}{$\mathrm{Poisson}(\lambda)$ Galton-Watson tree.} 
We recall that for $0 < \lambda  \le 1$, the random variable $|\pgw(\lambda)|$ has the {\em Borel-Tanner}$(\lambda)$ distribution, 
given by 
\begin{equation}\label{eq:borel-tanner}
\p{|\pgw(\lambda)|=m} = \frac{1}{m} \p{\mathrm{Poisson}(\lambda m)=m-1} = \frac{1}{m} \frac{e^{-\lambda m} (\lambda m)^{m-1}}{(m-1)!}\, .
\end{equation}
For $\lambda \ge 1$, write $\theta(\lambda)=\p{|\pgw(\lambda)|=\infty}$; 
\nomenclature[Theta]{$\theta(\lambda)$}{$\mathbf{P}\{\mathrm{PGW}(\lambda)=\infty\}$.} 
the function $\theta$ is strictly increasing, and for $\lambda > 1$ is infinitely differentiable and concave.\footnote{The concavity of $\theta$ can be verified using the implicit formula $\theta(\lambda)=e^{\lambda(\theta(\lambda)-1)}$.}
For $\lambda > 1$ we also write $B_{\lambda}$ for a random variable whose distribution is a
``truncated, size-biased'' analogue of the Borel-Tanner distribution: the distribution of $B_{\lambda}$ is given by 
\begin{equation}\label{eq:z_cond_dist}
\p{B_{\lambda} = m} 
=
\frac{\theta(\lambda)}{\theta'(\lambda)} \frac{e^{-\lambda m} (\lambda m)^{m-1}}{(m-1)!}\, .
\end{equation}
\nomenclature[Blambda]{$B_{\lambda}$}{``Truncated, size-biased Borel-Tanner random variable''.}
The fact that $\sum_{m \ge 1} \frac{\theta(\lambda)}{\theta'(\lambda)} \frac{e^{-\lambda m} (\lambda m)^{m-1}}{(m-1)!} = 1$, so that 
the preceding equation indeed defines a probability distribution, is proved in \citep[{Corollary~29}]{addario12prim}. 
Also, equation (\ref{eq:byform}), below, explains our use of the epithets 'truncated' and 'size-biased' for the random variable $B_{\lambda}$. 

By the results of \cite{addario12prim}, $((X_i,Z_i), i \ge 1)$ is a Markov process 
taking values in $(1,\infty) \times \N$, with infinitesimal generator $\kappa$ given by 
\[
\kappa((x,\ell),(y,m)) = \frac{\theta(y)}{\theta(x)} \frac{e^{-ym} (ym)^{m-1}}{(m-1)!} \I{y < x}\, .
\]
In other words, for $1 < y < x \le \infty$ and $1 \le \ell,m < \infty$ we have 
\[
\p{X_{n+1} \in dy ,Z_{n+1}=m~|~(X_n,Z_n)=(x,\ell)} = \frac{\theta(y) dy}{\theta(x)} \frac{e^{-ym} (ym)^{m-1}}{(m-1)!}\, .
\]
Equivalently (see \cite{addario12prim}, Lemma~28 and Corollary~29), we have 
\begin{align}
\p{X_{n+1} \in dy~|~X_n=x} & = \frac{\theta'(y) dy}{\theta(x)}\, , \label{eq:density}
\end{align}
and, conditional on $X_{n+1}$, the random variable $Z_{n+1}$ is distributed as $B_{X_{n+1}}$ and is (conditionally) independent of 
$((X_i,Z_i),i \le n)$. 
Furthermore, $X_1 \eqdist \theta^{-1}(U)$, where $U$ is Uniform$[0,1]$, and, conditional on $X_1$, the random variable $Z_1$ is distributed as $B_{X_1}$. Together with the above 
generator, this initial distribution specifies the distribution of the whole process $((X_i,Z_i),i \ge 1)$. We remark that for each $i \ge 1$, the distribution of $X_{i+1}$ given that $X_i=x$ is the same as the distribution of $X_1$ given that $X_1 \le x$. 

Conditional on the sequence $(Z_i,i \ge 1)$, independently for each $i \ge 1$, $P_i$ is distributed as a uniformly random labelled tree with $Z_i$ vertices. (Equivalently, $P_i$ is distributed as $\pgw(\lambda)$ conditioned to have $Z_i$ vertices - this conditional distribution does not depend on $\lambda > 0$; see, e.g., \cite{lyons12prob}, Exercise 5.15.) 
 Finally, for each $i \ge 1$, conditional on $P_i$, the vertices $R_i$ and $S_i$ are independent, uniformly random elements of $V(P_i)$, and the weights $\{W(e):e \in E(P_i)\}$ are independent and 
 uniform on $[0,X_i]$ (recall that $W(\{R_i,S_{i+1}\})=X_i$). 

The process $(X_i,i \ge 1)$ is sometimes dubbed the {\em forward maximal process} (for invasion percolation on the PWIT in \cite{addario12prim} and for invasion percolation on regular trees in \cite{angel08invasion}). In this paper, we instead use this term for the process $((X_i,Z_i),i \ge 1)$ as a matter of convenience.
\nomenclature[Xizi]{$((X_i,Z_i),i \ge 1)$}{The forward maximal process.} 
 
\subsection{The Poisson Galton-Watson aggregation process} \label{sec:pwag}

Recall the description of the random RWGs $(\rM(\lambda),\lambda \ge 1)$ from Section~\ref{sec:briefsketch}. The aim of this section is to explicitly describe the dynamics of the process $(\rM(\lambda),\lambda \ge 1)$. 

The ``percolation probability'' $\theta(x)$ was defined in Section~\ref{sec:wwlprim}. Given $c > 1$, we also define the Poisson Galton-Watson ``dual parameter'' $c^*$, which is the unique value $c' < 1$ for which 
$ce^{-c}=c' e^{-c'}$. (Another identity for $c^*$ which we will use later is that $c^*=c(1-\theta(c))$.) It is straightforward to verify that for $c > 1$, if $T$ has distribution $\pgw(c)$ then the conditional distribution of $T$, given that $T$ is finite, is $\pgw(c^*)$. We use that $c^*$ decreases as $c$ increases, which is also straightforward to check. 

For any node $v \in V(U)$, $v \ne \emptyset$, write $p(v)$ for the parent of $v$. \nomenclature[P(v)]{$p(v)$}{The parent of node $v$} Then temporarily write $U_v$ for the subtree of $U$ rooted at $v$ and containing only edges of weight less than $W(\{p(v),v\})$. Note that for $c > 1$, conditional on $W(\{p(v),v\})=c$, the tree $U_v$ has $\pgw(c)$ distribution. If we additionally condition $U_v$ to be finite, then it has distribution $\pgw(c^*)$. 

For each vertex $v \in V(M)$, there is a unique infinite path $P$ within $M$ leaving $v$. Write 
$x(v) = \sup\{W(e): e \in P\}$. 
\nomenclature[X(v)]{$x(v)$}{Weight of the largest weight edge in the unique infinite path in $M$ leaving $v$.} 
(If $v \in V(U)\setminus V(M)$ then set $x(v)=\infty$.) If $v \in V(T)$ then, in the terminology of the preceding section, we have $x(v)=X_i$ for some $i \ge 1$. If $v \not \in V(T)$ then consider the unique edge $e=\{u,u'\}$ of $P$ with $u \in V(T)$, $u' \not \in V(T)$. Since $T$ is built by invasion percolation (and is locally finite), there is an infinite path within $T$ leaving $u$ and with all edges of weight less than $x(v)$. In this case the (finite) portion of $P$ connecting $v$ and $u'$ contains a unique edge $e'$ with $a(v)=W(e')$, and we must then have $x(v)=W(e')=a(v)$. Note that the collection of values $(x(u),u \in V(M(\lambda)))$ is measurable with respect to $\rM(\lambda)$. 

Now fix $\lambda > 1$ and a vertex $u \in V(U)$. 
If $u \in V(M(\lambda))$ then let $e_u=\{u,u'\}$ be the smallest weight edge of $U$ incident to $u$ that is not contained in $M(\lambda)$. Necessarily, $u'$ is a child of $u$ in $U$. 
If $u \in V(T)$ then it is possible that $x(u) > \lambda$. In this case all edges from $u$ leaving $T$ have weight at least $x(u) > \lambda$. However, if $u \not\in V(T)$ then $x(u) =a(u) \le \lambda$. Whether or not $u \in V(T)$, given that $u \in V(M(\lambda))$ and $x(u) \le \lambda$, we have that $W(e_u)$ has distribution $\lambda + \mathrm{Exponential}(1)$. Furthermore, given that $u \in V(M(\lambda))$ and $x(u) \le \lambda$, the edge $e$ is an edge of $M$ precisely if $|U_v| < \infty$. 

By the above, we have that for small $\delta > 0$, 
\begin{align*}
& \quad\p{ W(e) \in (\lambda,\lambda+\delta], e \in E(M(\lambda+\delta]),u \in V(M(\lambda)),x(u) \le \lambda)~|~\rM(\lambda)} \\
=&\quad  (1+o_{\delta \downarrow 0}(1)) (1-e^{-\delta}) \p{|\pgw(\lambda)| < \infty}\cdot \I{u \in V(M(\lambda)),x(u) \le \lambda} \\
=&\quad (1+o_{\delta \downarrow 0}(1)) \delta (1-\theta(\lambda))\cdot \I{u \in V(M(\lambda)),x(u) \le \lambda}\, ,
\end{align*}
and 
\begin{align*}
& \quad \p{ W(e) \in (\lambda,\lambda+\delta], e \in E(M(\lambda+\delta]),u \in V(M(\lambda)),x(u) > \lambda~|~\rM(\lambda)} \\
= & \quad o_{\delta \downarrow 0}(1) \I{u \in V(M(\lambda)),x(u) > \lambda}\, .
\end{align*}
Furthermore, given that $W(e) \in (\lambda,\lambda+\delta]$ and that $e \in E(M(\lambda+\delta])$, 
the tree $U_v$ stochastically dominates $\pgw((\lambda+\delta)^*)$ and is stochastically dominated by $\pgw(\lambda^*)$, and $a(w) \in (\lambda,\lambda+\delta]$ for all vertices $w \in V(U_v)$ (in particular for $w=v$). Finally, as the edge weights on vertex-disjoint subtrees of $U$ are independent, the events 
\[
\pran{\{W(e_u) \in (\lambda,\lambda+\delta]\}\cap \{e_u \in E(M(\lambda+\delta))\},u \in V(M(\lambda))}
\]
are conditionally independent given $\rM(\lambda)$, and the subtrees $(U_{u'}, u \in V(M(\lambda)))$ are likewise conditionally independent. 

The preceding information yields the following description of the process $(M(\lambda),\lambda \ge 1)$. At each time $\lambda \ge 1$, there is a set of {\em active vertices} (those vertices $v \in V(U)$ with $x(v) \le \lambda$).  Once a vertex $v$ is active, Poisson Galton-Watson trees begin to attach themselves at $v$; at time $\lambda$, attachments occur at rate $(1-\theta(\lambda))$. This process happens independently for each active vertex. A tree attaching to $v$ at time $s$ is distributed as $\pgw(s^*)$. When a tree attaches to $v$, all its vertices immediately become active. 

We now add the edge weights to the above picture. The edge weights for $M(1)=T$ are already defined.
For $v \in V(U)$, $v \ne \emptyset$, the weights of edges in the subtree of $U$ rooted at $v$ are independent Exponential$(1)$ random variables, independent of all edge weights ouside this subtree. It follows that given that given that $x(v)=\lambda$ (i.e. that $v \in V(M(\lambda))$ and $v \not\in M(\lambda-)$), and given $U_v$, 
the weights of edges in $U_v$ are independent, with density $f_\lambda(x)=(1-e^{-\lambda})\cdot e^{-x}\I{0 \le x \le \lambda}$, and these edge weights are conditionally independent of all other weights of edges in $M(\lambda)$. 
In other words, if a tree $T$ attaches to vertex $u$ at time $\lambda$ then the edge $e$ from the root of $T$ to $u$ has weight exactly $\lambda$, and edges of $T$ have independent weights, each with density $f_\lambda(x)$.

Based on the above description, we call the RWG-valued process $(\rM(\lambda),1 \le \lambda < \infty)$ a {\em Poisson Galton-Watson aggregation process}. 
Note that we may recover $\rM$ as $\rM(\infty)=\lim_{\lambda \to \infty} \rM(\lambda)$. 
For a vertex $v \in V(M)$, if $v \not \in V(M(1))=V(T)$ then 
the value $x(v)$ is the time at which $v$ was added during the Poisson Galton-Watson aggregation process, so $v \in V(M(x(v)))$ but $v \not\in V(M(x(v)-))$. 
Furthermore, if $v \not\in V(T)$ then since a random tree attached at time $x(v)$ has distribution $\pgw(x(v)^*)$, 
the degree $\deg_{M(x(v))}(v)$ has distribution $\mathrm{Poisson}(x(v)^*)+1$, where the $1$ accounts for the parent of $v$. It follows that for any $v \in V(U)$, given that $v \in V(M)$, conditional on $\rT$ and on $x(v)$, the degree $\deg_{M}(v)$ of $v$ in $M$ has distribution
\[
\deg_{M}(v) \eqdist 
\begin{cases}
\deg_{T}(v) + \mathrm{Poisson}(\int_{x(v)}^{\infty} (1-\theta(t))\mathrm{d} t) & \mbox{if}~v \in V(T) \\
1 + \mathrm{Poisson}(x(v)+\int_{x(v)}^{\infty} (1-\theta(t)) \mathrm{d} t)	& \mbox{if}~v \not\in V(T). 
\end{cases}
\]
Since $\int_1^{\infty} (1-\theta(t)) < \infty$, it follows that for all $v \in V(T)$, the conditional 
expectation 
\[
\E{\deg_{M}(v)~|~v \in V(M),x(v),\rT}
\]
 is almost surely finite; from this it is easily seen that $M$ is almost surely locally finite. Also, it will follow from our arguments for the upper bound in Theorem~\ref{thm:volume} that $M$ is almost surely one-ended (see Corollary~\ref{cor:nvdist}, below; there should be a simple direct proof of this fact, however). 

We will eventually prove Theorem~\ref{thm:main} by showing that a hybrid construction process for $\rM_n$, defined in Section~\ref{sec:mainproof}, converges in distribution to the Poisson Galton-Watson aggregation process. 

\subsection{A few consequences of Theorem~\ref{thm:main}.}\label{sec:consequences}
We now briefly remark on some consequences of our weak convergence result and of the above descriptions of\, $\rT$ and $\rM$. The characterization given by Theorem~\ref{thm:main} is straightforward enough to yield an explicit, though somewhat complicated, description of the distribution of the 
degree of the root vertex $\emptyset$. First, as in Section~\ref{sec:wwlprim}, let $U$ be Uniform$[0,1]$, let $X_1=\theta^{-1}(U)$, and given the value of $X_1$, let $Z_1$ be distributed as $B_{X_1}$, whose distribution is given in (\ref{eq:z_cond_dist}).

For integer $m \ge 1$, let $\nu_m$ be the distribution of the degree of node $1$ in a uniformly random labelled tree with nodes $\{1,\ldots,m\}$; this distribution is explicitly given by 
\[
\nu_m(k) = k \cdot \frac{(m-k)^{m-k-1}}{m^{m-1}} \cdot {m-2 \choose k-1}. 
\]
Let $D_1$ have distribution $\nu_{Z_1}(k)$; in other words, given that $Z_1=m$, $D_1$ has distribution $\nu_m$. Also, 
let $D_2$ have distribution Bernoulli$(1/Z_1)$, with $D_2$ conditionally independent of $D_1$ given $(X_1,Z_1)$. The description of $\rT$ given in Section~\ref{sec:wwlprim} implies that $\deg_{T}(\emptyset)$ has distribution $D_1+D_2$. 

Finally, the description of the Poisson Galton-Watson aggregation process implies 
that given $(X_1,Z_1)$, $\deg_M(\emptyset)-\deg_T(\emptyset)$ is independent 
of $\det_T(\emptyset)$ and is Poisson$(\int_{X_1}^{\infty} (1-\theta(t)) \mathrm{d}t)$ distributed. 
Letting $D_3$ be Poisson$\int_{X_1}^{\infty} (1-\theta(t)) \mathrm{d}t$ and be 
conditionally independent of $D_1$ and of $D_2$ given $(X_1,Z_1)$, we then have 
the following corollary. 
\begin{cor}
The random variable $\deg_{M}(\emptyset)$ is distributed as $D_1+D_2+D_3$. 
\end{cor}
It is interesting to contrast this result with Proposition~2 from \cite{aldous90randomtree}, which expresses 
$\deg_{M}(\emptyset)$ as a mixture of Poisson random variables: 
\[
\p{\deg_{M}(\emptyset)=i+1} = \int_0^1 \p{\mathrm{Poisson}(\Phi(u))=i}\mathrm{d}u\, ,
\]
where for $0\le u < 1$ we set $\phi(u)=\int_0^u \log(1/x)/(1-x)\mathrm{d}x$. 

We close this section with a final observation. By a classic result of Frieze \cite{frieze85mst} we have that 
\[\E{\sum_{e \in E(M):\emptyset \in e} W(e)} = 2\zeta(3)\, .
\] 
Combining this fact with Theorem~\ref{thm:main} and with Theorem~3 of \cite{mcdiarmid97mst}, it follows that 
for $U\eqdist \mathrm{Uniform}[0,1]$, we have 
\[
\E{\int_{\theta^{-1}(U)}^{\infty} t(1-\theta(t)) \mathrm{d}t} = 2\zeta(3)-\zeta(2)\, ,
\]
which is perhaps not obvious.  

\section{ {\bf Future maxima in $\rK_n$, and a strengthening of Proposition~\ref{thm:abgk}}}\label{sec:futuremax}
The paper \cite{addario12prim} shows that the local weak limit of the early stages of Prim's algorithm on $\rK_n$ is given by $\rT=(T,\emptyset,W)$.\footnote{See the remark just after Theorem~27 of \cite{addario12prim}, together with the remark in Section~1.1 of the same paper.}
More precisely, for $1 \le k \le n$ let $M_{n,k}$ be the subtree of $M_n$ built by the first $k$ steps of invasion percolation on $\rK_n$, so $M_{n,k}$ has vertices $1=v_1(\rK_n),\ldots,v_k(\rK_n)$ and edges $e_1(\rK_n),\ldots,e_{k-1}(\rK_n)$. 
\nomenclature[Mnk]{$M_{n,k}$}{Subtree of $M_n$ built by first $k$ steps of invasion percolation.} 
Then write $\rM_{n,k}=(M_{n,k},1,W_n)$. 
Similarly, 
for $k \ge 1$, write $T_k$ for the tree 
built by the first $k$ steps of Prim's algorithm on $\rU=(U,\emptyset,W)$, so $T_k$ has vertices $\emptyset=v_1(\rU),\ldots,v_k(\rU)$ and edges $e_1(\rU),\ldots,e_{k-1}(\rU)$, 
and let $\rT_k=(T_k,\emptyset,W)$. 
\nomenclature[Tk]{$T_{k}$}{Subtree of $T$ built by first $k$ steps of invasion percolation.} 
\begin{prop}[\cite{addario12prim}]\label{thm:abgk}
Fix any function $k(n):\N \to \N$ such that $k(n) \to \infty$ and $k(n)=o(n^{1/2})$. 
Then for each $n$ we may couple $\rM_{n,k(n)}$ and $\rT$ so that 
\[
\p{\rM_{n,k(n)} \ne \rT_{k(n)}} \to 0
\]
as $n \to \infty$. In particular, $\rM_{n,k(n)} \convdist \rT$ in the local weak sense. 
\end{prop}
The RWG $\rT$ is almost surely locally finite (this is proved in \cite{addario12prim} but is also easy to see directly), and so necessarily $\rT_{k}\convas \rT$ in the local weak sense, as $k \to \infty$. 
Using this fact, the first assertion of the proposition immediately yields the distributional convergence claimed in the proposition. 
We will in fact require an analogue of Proposition~\ref{thm:abgk} that holds as long as $k(n)=o(n)$, and now prove such a result. 
\begin{prop}\label{prop:abgw}
Fix $f:\N\to\N$ with $f(n) \to \infty$ and $f(n)=o(n)$. Then we may couple $\rM_{n,f(n)}$ and $\rT$ so that 
\[
\lim_{n \to \infty} \sup\left\{r: \rM_{n,f(n)}[B_{\rM_{n,f(n)}}(r)] = \rT_{f(n)}[B_{\rT_{f(n)}}(r)] \right\} \stackrel{\mathrm{a.s.}}{=} \infty. 
\]
In particular, $\rM_n(f(n)) \convdist \rT$
as $n \to \infty$. 
\end{prop}
The proof of Proposition~\ref{prop:abgw} will introduce, in simplified form, some of the important 
structures and techniques that will be developed later in the paper. 

We begin by considering Prim's algorithm on the PWIT. 
Given $v \in V(U)$ and $z > 0$, write $U_v(z)$ for 
the subtree of $U$ consisting of $v$ and all descendents of $v$ whose 
path to $v$ contains only edge of weight less than $z$. (In the notation of Section~\ref{sec:pwag}, 
for $v \ne \emptyset$, we in particular have $U_v(W(\{p(v),v\}))=U_v$.)
This 
tree is is distributed as $\pgw(\lambda)$ and is therefore infinite with probability $\theta(z)$.

Given $z > 1$, let 
\[
g(z)=\max\{i: W(e_i(\rU)) \ge z\}\, ,
\]
or $g(z)=0$ if $W(e_i(\rU)) < z$ for all $i \ge 1$. 
Next, list the indices $i$ for which $e_i(\rU) \ge z$ as 
$(i_1,\ldots,i_J)$, where $J=J(z,\rU) \ge 0$ is random. 
Note that for each $j \ge 1$, if $J \ge j$ then 
$U_{v_{i_j+1}(\rU)}(z)$ is distributed as $\pgw(z)$,
and so 
\[
\p{J=j|J \ge j} = \theta(z)\, .
\]
 It follows that for all $j\ge 0$, $\p{J=j} = (1-\theta(z))^{j}\theta(z)$, 
 and $\e{J}=1/\theta(z)$. 
Also note that for $j > 0$, if $J > j$ then $U_{v_{i_j+1}}(z)$ is distributed 
as $\pgw(z)$ conditioned to be finite, or in other words 
is distributed as $\pgw(z^*)$, and so 
\[
\E{i_{j+1}-i_j|J > j} = \E{|\pgw(z^*)|} = \frac{1}{1-z^*}\, 
\]
(the latter equality, easily proved, is re-stated in (\ref{eq:meanvar}) below). 
Since $g(z)=i_J$ whenever $g(z)>0$, it follows that 
\begin{equation}\label{eq:gzexp}
\E{g(z)} = \frac{1}{\theta(z)}\cdot\frac{1}{1-z^*}. 
\end{equation}

We next turn to Prim's algorithm on $\rK_n$. Using the same definition of $g(z)$ does not turn 
out to be suitable in this case: with high probability the largest weight edge is the 
last edge added by Prim's algorithm, in which case we would have $g(z)=n-1$ for all $z$ for which 
there are edges of weight greater than $z$ in the minimum spanning tree. We instead look at the largest 
edge added before some 'threshold time' $k$; this is formalized in the following definition. 

For $z \ge 0$ let $K_n^z$ be the subgraph of $K_n$ with edges $E(K_n^z)=\{e \in E(K_n): W_n(e) \le z\}$. 
\nomenclature[Knz]{$K_n^z$}{Subgraph of $K_n$ with edges $E(K_n^z)=\{e \in E(K_n): W_n(e) \le z\}$.}
Given $1 \le j \le n-1$ and $z > 0$, let 
\begin{equation}\label{eq:gzdef}
g_n(j,z)= \sup\{ \ell: 1 \le \ell \le j, W_n(e_{\ell}(\rK_n)) \ge z\}\ ,
\end{equation}
In words, $g_n(j,z)$ is the last time up to step $j$ that Prim's algorithm on $\rK_n$ adds an edge of weight at least~$z$. 
\nomenclature[Gn]{$g_n(j,z)$}{Last  time before step $j$ that Prim's algorithm on 
$\rK_n$ adds an edge of weight~$\ge z$.}
We wish to bound the upper tail of $g_n(m,z)$, for suitable values of $m$ and $z$. 
To do so, list the connected components of $K_n^z$ as $C_n^z=(C_{n,1}^z,\ldots,C_{n,m}^z)$, 
in decreasing order of size (with ties broken lexicographically, say). 
\nomenclature[Cnz]{$C_n^z$}{Components of $K_n^z$, listed in decreasing order of size as $(C_{n,1}^z,\ldots,C_{n,m}^z)$.}
Consider the behaviour of Prim's algorithm conditional on the sequence $C_n^z$. 
Each time Prim's algorithm adds an edge connecting to some component $C_{n,i}^z$, it 
fully connects $C_{n,i}^z$ before adding any edge leaving $C_{n,i}^z$. Furthermore, once $C_{n,i}^z$ 
is fully connected, the edge leaving $C_{n,i}^z$ connects $C_{n,i}^z$ with a uniformly random vertex 
among all those not yet uncovered. Writing $\sigma_n=\sigma_n(z)$ for the number of components 
uncovered before an edge to $C_{n,1}^z$ is added, for each $j \ge 0$ we therefore have 
\[
\p{\sigma_n = j~|~ \sigma_n \ge j, C_n^z} \ge \frac{|C_{n,1}^z|}{n}. 
\]
We next use two bounds for the sizes of the components of $K_{n}^{z}$; the first is due to Stepanov \cite{stepanov1970probability} (pages 64-65), and the second can be found in \cite{remco13rgcn}, Proposition 4.12.\footnote{There is probably an earlier reference for this rather straightforward result, but I have not managed to track one down.} 
For any $\eps > 0$ there is $\delta > 0$ such that for all $n$ sufficiently large, 
\begin{equation}\label{eq:giantldbound}
\p{|C_{n,1}^z-n\theta(z)| > \eps n} < e^{-\delta n}, \quad \p{|C_{n,2}^z| \ge \log^3 n} \le n^{-100}\, .
\end{equation}
It follows from the first bound that there is $\delta > 0$ such that for all $j \ge 1$, 
\[
\p{\sigma_n > j} \le (1-\theta(z)/2)^j + e^{-\delta n}. 
\]
However, for $m \ge \sigma_n \cdot |C_{n,2}^z|$, if $|C_{n,1}^z| \ge m$ then necessarily $g_n(m,z) \le \sigma_n \cdot |C_{n,2}^z|$. 
It then follows from the above bounds that if $\log^5 n \le  m \le n(\theta(z)-\eps)$, we have
\begin{align}
\p{g_n(m,z) \ge \log^5 n} & 
\le \p{|C_{n,2}^z| \ge \log^3 n} + \p{\sigma_n \ge \log^2 n} + \p{|C_{n,1}| \le m} \nonumber\\
& \le n^{-100}+ (1-\theta(z)/2)^{\log^2 n} + 2e^{-\delta n} \nonumber \\
& \le n^{-99}\, ,\label{eq:gzupper}
\end{align}
for $n$ sufficiently large. 
With these bounds under our belt, we are prepared to prove Proposition~\ref{prop:abgw}.
\begin{proof}[Proof of Proposition~\ref{prop:abgw}]
\begin{figure}[htb]
\includegraphics[width=0.8\textwidth]{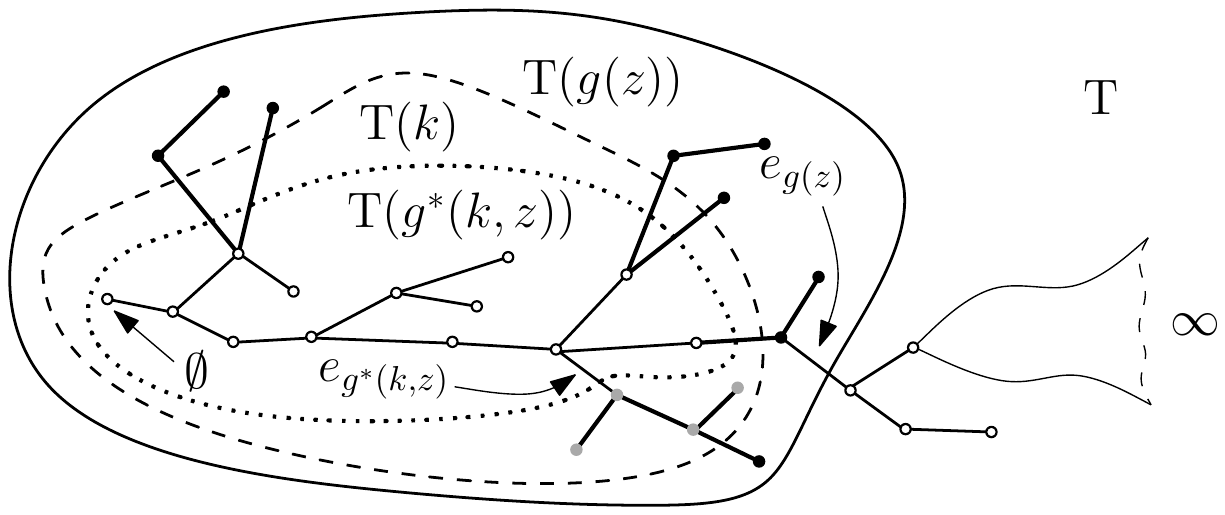}
\caption{
The trees $\rT_{g(z)}$, $\rT_{k}$, and $\rT_{g^*(k,z)}$ are circled with solid, dashed, and dotted lines, respectively. 
Vertices of $\rT_{g(z)}$ not in $\rT_k$ are drawn as solid black dots, and vertices of 
$\rT(k)$ not in $\rT(g^*(k,z))$ are drawn as solid grey dots. Note that the latter form a subtree of $\rT_k$ rooted 
at the far endpoint of $e_{g^*(k,z)}$.
}
\label{fig:abgwdef}
\end{figure}

The definitions of the current paragraph are for the most part illustrated in Figure \ref{fig:abgwdef}. 
For integer $k \ge 1$, 
let $g^*(k,z)=g(k,z,\rT_k)$, so $g^*(k,z)$ is the last time 
Prim's algorithm on $\rU$, started from $\emptyset$, uncovers an edge of weight at least $z$ before time $k$.
We always have $g^*(k,z) \le g(z)$ and so $T_{g^*(k,z)}$ is a subtree of $T_{g(z)}$.  
Furthermore, if  $T_{g^*(k,z)} \ne T_{g(z)}$ then $g(z) > k$ and 
so $T_k$ is a subtree of $T_{g(z)}$. In the latter case, the vertices of 
$T_k$ not in $T_{g^*(k,z)}$ form a subtree of $T_k$ whose path to 
the root passes through $e(g^*(k,z))$. On the other hand, vertices 
of $T_{g(z)}$ not in $T_k$ may in principle attach to any point of $T_{g(z)}$. 

Since $\rT_{g^*(k,z)} \ne \rT_{g(z)}$ implies that $g(z) > k$, 
by (\ref{eq:gzexp}) and Markov's inequality we have 
\begin{equation}\label{eq:abgz1}
\p{\rT_{g^*(k,z)} \ne \rT_{g(z)}} \le \frac{1}{k\theta(z)(1-z^*)}\, .
\end{equation}
Our next aim is to prove an analogous result in the finite setting. 
Recall the definition of $g_n(z,k)$ from (\ref{eq:gzdef}). 
Fix functions $f(n)$ and $k(n)$ with $\log^5 n \le k(n) < f(n)$, with $k(n)=o(\sqrt{n})$, and with $f(n)=o(n)$. 
The quantities $g_n(k(n),z)$ and $g_n(f(n),z)$ will play the roles of $g^*(k,z)$ and $g(z)$, respectively. 

For fixed $z> 1$, if $\rM_{n,g_n(k(n),z)}\ne \rM_{n,g_n(f(n),z)}$, 
then Prim's algorithm adds an edge of weight at least $z$ 
at some step $i$ with $k(n) < i \le f(n)$. The latter implies 
that $g_n(f(n),z) > k(n)$, and since $k(n) \ge \log^5 n$, by (\ref{eq:gzupper}) we then have 
\begin{equation}\label{eq:abgz2}
\p{\rM_{n,g_n(k(n),z)}\ne \rM_{n,g_n(f(n),z)}} \le \p{g_n(f(n),z) > k(n)} \le n^{-99}\, ,
\end{equation}
for $n$ large enough. 

By Proposition~\ref{thm:abgk}, we may couple $\rM_{n,k(n)}$ and $\rT_{k(n)}$ so that 
\[
\p{\rM_{n,k(n)} \ne \rT_{k(n)}} \to 0\, .
\]
Also, whenever $\rM_{n,k(n)} = \rT_{k(n)}$, we have $g^*(k(n),z)=g_n(k(n),z)$ and so 
$\rT(g^*(k(n),z))=\rM_{n,g_n(k(n),z)}$, and it follows that 
\begin{equation}\label{eq:abgz3}
\p{\rT_{g^*(k(n),z)} \ne \rM_{n,g_n(k(n),z)}} \to 0\, .
\end{equation}
Combining (\ref{eq:abgz1}), (\ref{eq:abgz2}), and (\ref{eq:abgz3}), we obtain that 
\[
\p{\rT_{g(z)} \ne \rM_{n,g_n(f(n),z)}} \le \frac{1}{k(n)\theta(z)(1-z^*)} + \frac{1}{n^{99}}+o(1),
\]
which tends to zero as $n \to \infty$, for any fixed $z > 1$. 
Finally, fix an integer $r > 1$. 
If $\sup\{i: X_i > z\} \ge r$ then $B'_{\rT}(\emptyset,r) \subset V(T_{g(z)})$, and it 
follows that 
\[
\p{B'_{\rT}(\emptyset,r) \subseteq V(T_{g(z)})} \ge 1-\p{\sup\{i: X_i > z\} < r}\, .
\]
By the expression (\ref{eq:density}) for the conditional densities of the $X_i$ and the fact that $\theta'(1+\eps)=2+o(\eps)$ as $\eps \downarrow 0$, it easily follows that 
\[
\p{\sup\{i: X_i > 1+3^{-r}\} < r} \to 0
\]
as $r \to \infty$. 
It follows that for any $\eps > 0$, for all $r$ sufficiently large, writing $z=1+1/3^r$ we have 
$\p{\sup\{i: X_i > z\} < r} < \eps$, and for such $r$ 
we obtain that 
\begin{align*}
\p{\rM_{n,f(n)}(r) \ne \rT_{f(n)}(r)} & \le \p{B'_{\rT}(\emptyset,r) \not\subseteq V(T_{g(z)})} 
+ \p{\rT_{g(z)} \ne \rM_{n,g_n(f(n),z)}}  \\
& = \eps + o_{n \to \infty}(1)\, .
\end{align*}
As $\eps > 0$ was arbitrary, this completes the proof. 
\end{proof}

\section{ {\bf Properties of the forward maximal process $((X_n,Z_n),n \ge 1)$}}\label{sec:formax}

\subsection{The tails of the random variables $\{B_y,y > 1\}$.} 
\hspace{0.3cm}

\vspace{0.1cm} \noindent
For $0 < \lambda < 1$, the first two moments of $|\pgw(\lambda)|$ are given by the following simple formulas \cite{borel42emploi}: 
\begin{equation}\label{eq:meanvar}
\E{|\pgw(\lambda)|} = \frac{1}{1-\lambda},\quad \E{|\pgw(\lambda)|^2} = \frac{1}{(1-\lambda)^3}\, .
\end{equation}
Also, from (\ref{eq:borel-tanner}) and Stirling's formula, it follows that 
\[
\p{|\pgw(\lambda)|=m} \sim \frac{e^{1-\lambda}}{(2\pi)^{1/2}}\cdot \frac{1}{m^{3/2}} \cdot (\lambda e^{(1-\lambda)})^{m-1}\, ,
\]
as $m \to \infty$, uniformly in $0 \le \lambda \le 1$. Using explicit error bounds for Stirling's approximation, it is not hard to see that in fact, 
for all $m \ge 1$ and $0 \le \lambda \le 1$, 
\[
\frac{1}{3m^{3/2}} \cdot (\lambda e^{(1-\lambda)})^{m-1} \le \p{|\pgw(\lambda)|=m} \le \frac{3}{m^{3/2}} \cdot (\lambda e^{(1-\lambda)})^{m-1}\, .
\]
We leave the detailed verification of these inequalities to the reader. 
Using that $\lambda e^{1-\lambda} \le e^{-(1-\lambda)^2}$ for all 
$0 \le \lambda \le 1$, 
and that $\lambda e^{1-\lambda} \ge e^{-2(1-\lambda)^2}$ for $\lambda$ sufficiently close to $1$, it follows
for all $m \ge 1$ and for all $0 \le \lambda \le 1$, 
\begin{equation}\label{eq:pgwlambdabound}
\p{|\pgw(\lambda)|=m} \le \frac{3}{m^{3/2}} \cdot e^{-m(1-\lambda)^2}. 
\end{equation}

Next recall (from Section~\ref{sec:pwag}) the definition of the {\em dual parameter} $u^*$: for $u > 1$, $u^*$ is the unique 
value $\lambda \in (0,1)$ with $ue^{-u}=\lambda e^{-\lambda}$. It is straightforward to check that $\pgw(u)$, 
conditional on the event that $|\pgw(u)|<\infty$, has the same distribution as $\pgw(u^*)$. 
It is also easily checked that $(1+\eps)^*=1-\eps+o(\eps)$ as $\eps \downarrow 0$. 

Fix $y> 1$ and let $B_y$ be as in (\ref{eq:z_cond_dist}). 
We next state probability bounds for the upper and lower tails of $B_y$ 
when $y$ is near $1$. 
\begin{lem}\label{lem:bybounds}
There is $y_0 > 1$ such that for all $1 < y < y_0$ and any constants $0 < c < 1 < C$, we have 
\[
\p{B_y \le \frac{c}{(y-1)^2}} \le 8c^{1/2}, \qquad \p{B_y \ge \frac{C}{(y-1)^2}} \le 20e^{-C/8}. 
\]
\end{lem}
\begin{proof}
Writing $c_y = \theta(y) y^*/ y \theta'(y)$, 
straightforward calculation shows that we may rewrite the probability $\p{B_y=k}$ as 
\begin{align}
\p{B_y=k} & = \frac{\theta(y) y^*}{ y \theta'(y)} \cdot k \cdot \frac{e^{-y^*k}(y^*k)^k-1}{k!} \nonumber\\
		& = c_y \cdot k \cdot \p{|\pgw(y^*)|=k}\, .\label{eq:byform}
\end{align}
Since $\pgw(y^*)$ is distributed as $\pgw(y)$ conditioned to be finite, (\ref{eq:byform}) may explain why we referred 
to the law of $B_y$ as a truncated, size-biased version of the law of $\pgw(y)$. 

The above asymptotics for $\theta(y)$, $\theta'(y)$ and $y^*$ near $y=1$ yield that 
$c_y \sim y-1$ as $y \downarrow 1$. 
The above upper bound for the tails of the random variables $|\pgw(\lambda)|$ 
then implies that for all $y \ge 1$ and $k \ge 1$, 
\[
\p{B_y = k} \le \frac{3c_y}{k^{1/2}} e^{-k(1-y^*)^2}
\]
Using (\ref{eq:byform}) and the $y \downarrow 1$ asymptotic for $c_y$, we obtain that for $y$ sufficiently close to $1$ and any 
$0 < c < 1$, 
\begin{align}
\p{B_y \le \frac{c}{(y-1)^2}} & \le 3 c_y \sum_{i \le c/(y-1)^2} \frac{1}{i^{1/2}} \nonumber\\
					& < 7 c_y (\lfloor c/(y-1)^2\rfloor)^{1/2} \nonumber \\
					& < 8 c^{1/2}\, .\label{eq:bysizelb}
\end{align}
Similarly, assuming $y_0$ is chosen small enough that $(1-y_0^*)/(y_0-1) \ge 1/2$, 
we have that for all $1 < y < y_0$ and for all $C > 1$, 
\begin{align}
\p{B_y \ge \frac{C}{(y-1)^2}} & \le 3 c_y \sum_{i \ge C/(y-1)^2} \frac{1}{i^{1/2}} \exp\pran{-\frac{i}{2(y-1)^2}} \nonumber \\
					& < 7c_y (\lfloor C/(y-1)^2\rfloor)^{1/2} \frac{e^{-C/2}}{(1-e^{-C/2})^2} \nonumber \\
					& < 20 C^{1/2} e^{-C/2} < 20 e^{-C/8}\, . \label{eq:bysizeub}
\end{align}
\end{proof}

\subsection{The growth and decay of $(X_n,n \ge 1)$ and of $(Z_n, n \ge 1)$.}\label{sec:halving}

In this section we state three lemmas that will be used in 
Recall from Section~\ref{sec:wwlprim} that for $n \ge 1$, the conditional density of $X_{n+1}$ given 
that $X_n=x$ is given by $f(y)=\theta'(y)/\theta(x)$ for $y \in (1,x)$, where 
$\theta(\lambda)=\p{|\pgw(\lambda)|=\infty}$. 
In other words, under this conditioning, $X_{n+1}$ is distributed as $X_1$ conditioned to satisfy $X_1 \le x$. 
The function $\theta$ satisfies 
\[
\theta'(1+\eps) = 2 + o(\eps),
\]
as $\eps \downarrow 0$, and so $\theta(1+\eps) = 2\eps + o(\eps)$. It follows that for large $n$, the ratios $(X_{n+1}-1)/(X_n-1)$ are approximately distributed as Uniform$[0,1]$ random variables, and so typically the difference $(X_{n}-1)$ should decrease by a factor two after increasing $n$ by a constant amount. The next lemma bounds the probability of seeing ``halving times'' that are substantially longer. 

For $z > 1$, let $I(z) = \min\{i: X_i \le z\}$. \nomenclature[Iz]{$I(z)$}{Index of first forward maximal weight $X_i$ with $X_i \le z$.}
Let $x_1>1$ be such that $\theta'(x_1)=1$; 
since the function $\theta$ is concave, such $x_1$ is unique. 
\nomenclature[X1]{$x_1$}{Unique value for which $\theta'(x_1)=1$.}
For $\eps > 0$, write 
\[
H_{\eps} = I(x_1) \vee \max \{k: \exists n, 1+\eps \le X_{n+k} < x_1, X_{n+k}-1 \ge (X_{n+1}-1)/2\}\, .
\]
In words, $H_{\eps}$ is the greatest number of steps required for the difference $(X_n-1)$ to fall below $x_1$, or to reduce 
by a factor of two once below $x_1$, before $X_n$ drops below $1+\eps$. The next lemma provides probability bounds 
on the upper tail of $H_{\eps}$. 

\begin{lem}\label{lem:halvingtimes}
There is an absoute constant $\cc\ccdef\ccht > 0$ such that for all $\eps > 0$ and $k > 1$, we have 
\[
\p{H_{\eps} > k} < \log_2(2/\eps) \cdot e^{-\ccht k}. 
\]
\end{lem}
\begin{proof}
First, for each $i \ge 1$, we have 
\[
\p{I(x_1) > i} \le \prod_{j=1}^i \sup_{x > x_1} \p{X_{i+1} > 2 | X_i=x} \le \prod_{j=1}^i \p{X_1 > x_1}  \le (1-\theta(x_1))^{i}. \\
\]
Now fix $\eps >0$ and $k > 1$. By our choice of $\ccht$ we may assume that $\eps$ is small, and in particular that $\eps < x_1-1$. 
Note that since $\theta$ is concave, $\theta(x_1) \ge (x_1-1)\theta'(x_1)=x_1-1$ and so necessarily $x_1 \le 1$. 
It follows that if $H_{\eps} > k$ then either $I(x_1) > k/2$ or else for some $1 \le i \le \log_2(1/\eps)$ we have 
\begin{equation}\label{eq:preceq}
\#\{j: \theta(x_1)/2^{i+1} < X_j \le \theta(x_1)/2^i\} > k/2\, .
\end{equation}
Temporarily write $j_0$ for the first $j$ for which $X_j-1 \le \theta(x_1)/2^i$. If (\ref{eq:preceq}) is 
to hold then we must in particular have $X_{j_0+\ell+1}-1 > (X_{j_0+\ell}-1)/2$ for each $1 \le \ell \le k/2$. 
By the Markov property the probability of the event in (\ref{eq:preceq}) is therefore at most 
\[
\left(\sup_{0 < x \le x_1-1} \p{X_{j+1}-1 \ge x/2~|~X_j-1=x}\right)^{k/2}\, .
\]
Since $x \le x_1$ and $\theta'(x_1)=1$, by convexity, for all $1 < x' \le x$ we have $1 \le \theta'(x) < 2$. Since the conditional density of $X_{i+1}$ at $x'$ given that $X_i=x$ is proportional to $\theta'(x')$, it follows that $\p{X_{i+1}-1 \ge x/2~|~X_i-1=x} \le 2/3$, and so 
\[
\p{\#\{j: \theta(x_1)/2^{i+1} < X_j \le \theta(x_1)/2^i\} > k/2}\le (2/3)^{k/2}\, .
\]
By a union bound it follows that 
\[
\p{H_{\eps} > k} \le (1-\theta(x_1))^{k/2} + \log_2(1/\eps)(2/3)^{k/2},
\]
and takng $\delta$ small enough that $e^{-\delta/2} > \max(1-\theta(x_1),2/3)$ completes the proof. 
\end{proof}
Later, we will also need the following tail bound on the total number of vertices in trees whose forward maximal weight is above a given threshold. 
\begin{lem}\label{lem:zisizes}
There exist constants $\cc\ccdef\cczi,\CC\CCdef\CCzi > 0$ such that 
for all $r > 1$ and $x > 1$, 
\[
\p{\sum_{i: X_i > 1+1/r} Z_i > xr^2} \le \CCzi \log r\cdot e^{-\cczi x^{1/2}}\, .
\]
\end{lem}
\begin{proof} 
By adjusting the value of $\CCzi$ we may assume that $x$ is at least 
$j > 12 \log_2 (\theta(x_1)\cdot r)$. 
In the proof of \refL{lem:halvingtimes} we showed that 
$\p{I(x_1) \ge i} \le (1-\theta(x_1))^i$ 
and that 
\[
\inf_{0 < x \le x_1-1} \p{X_{j+1}\le 1+x/2~|X_j=1+x} \ge 1/3\, .
\]
Writing $i_0=I(1+1/r)$, it follows from the above results that for $j \ge 1$, 
\[
\p{i_0 \ge j} \le \p{I(x_1) \ge j/2} + \p{\mathrm{Bin}(\ceil{j/2},1/3) < \log_2 (\theta(x_1)\cdot r)} < (1-\theta(x_1))^{j/2} + e^{-j/48}\, ,
\]
the last inequality holding for $j > 12 \log_2 (\theta(x_1)\cdot r)$. 
It follows that for such $j$, 
\begin{align*}
\p{\sum_{1 \le i < i_0} Z_i > xjr^2}	& \le \p{i_0 > j} + j \p{B_{1+1/r} > xr^2} \\
							& \le (1-\theta(y))^{j/2} + e^{-j/48} + 20j e^{-x/8}\, ,
\end{align*}
the last inequality holding by the upper bound in \refL{lem:bybounds}. 
Taking $j=x$ then proves the lemma. 
\end{proof}

Finally, the following lemma, 
which establishes bounds on the lower tail of $d_{\rT}(\emptyset,R_{i_0})$ and of $X_{i_0}$, 
will be used in proving the upper bound from Theorem~\ref{thm:volume}. 
In its proof, we will use the following explicit formula. Let $\cT^{(n)}$ be a uniformly random tree with vertices 
$\{1,\ldots,n\}$, and let $v_1,v_2$ be independent, uniformly random elements of $\{1,\ldots,n\}$. Then for each $1 \le k \le n-1$ we have 
\begin{equation}\label{eq:cayleydist}
\p{d_{\cT^{(n)}}(v_1,v_2) \ge k} = \prod_{j=1}^k \frac{n-j}{n}\, .
\end{equation}
\begin{lem}\label{lem:pond-dist-lower}
There exists $\CC\CCdef\CCdistlb>0$ such that 
for all $r>1$ with $1+1/r < x_1$, and all $x > 0$, writing $i_0=I(1+1/r)$, we have 
\[
\p{d_{\rT}(\emptyset,R_{i_0}) < xr \mbox{ or } X_{i_0} < x/r} < \CCdistlb x^{2/3}\, .
\]
\end{lem}
\begin{proof}
First, given $i_0$, the density of $X_{i_0}$ at $u \in (1,1+1/r)$ is $\theta'(u)/\theta(1+1/r)$. Since $\theta'(u) \in [1,2]$ for all $u \in (1,1+1/r)$, 
it follows that for all $0 < \eps < 1$ we have $\p{X_{i_0} < 1+\eps/r} < 2\eps$. 

Next, by the lower bound in \refL{lem:bybounds} we immediately have 
\[
\p{Z_{i_0} \le x^2r^2} \le 8x. 
\]
On the other hand, $(P_{i_0},R_{i_0},S_{i_0})$ is distributed as a uniformly random tree, together with two independent, uniformly random vertices, conditional on its size $Z_{i_0}$. By (\ref{eq:cayleydist}) we thus have 
\begin{align*}
\p{Z_{i_0} > x^2r^2, d_{P_{i_0}}(R_{i_0},S_{i_0}) \le cxr} 
& \le 1- \prod_{j=1}^{\floor{cxr}} \frac{x^2r^2-j}{x^2r^2} \\
& \le 1- \exp\pran{-2\sum_{j=1}^{\floor{cxr}} j/(x^2r^2)} \\
& \le 1-e^{-c^2} \\
& \le c^2\, ,
\end{align*}
the second inequality holding as long as $c < 1/4$, say. For $c>0$ sufficiently small, by taking $x=c^{2}$, it follows from these bounds that 
\[
\p{d_{P_{i_0}}(R_{i_0},S_{i_0}) \le c^3r} \le 9c^{2}\, .
\]
Since $d_{\rT}(\emptyset,R_{i_0}) \ge d_{P_{i_0}}(R_{i_0},S_{i_0})$, the result follows. 
\end{proof}

\section{ {\bf Volume growth in $\rT$: a proof of Theorem~\ref{thm:tvolume}}}\label{sec:volgor}
In the preceding section, Lemma~\ref{lem:zisizes} proved upper tail bounds for the total size of the ''forward maximal clusters'' added by invasion percolation before a given forward maximal edge. In order to prove Theorem~\ref{thm:tvolume}, we need a similar bound for the total {\em diameter} of such clusters. We first prove the requisite bound, then turn to the proof of Theorem~\ref{thm:tvolume}. 

\subsection{The diameters of the trees $(P_i,i \ge 1)$}
The subtrees $(P_i,i \ge 1)$ of $T$ were defined in Section~\ref{sec:wwlprim}. 
For $i \ge 1$, the tree $P_i$ is distributed as a uniformly random labelled tree with $Z_i$ vertices. 
A variety of authors \cite{szekeres83dist,flajolet92height,luczak95trees,addario12tail} have studied the tail behavior of the diameter of uniformly random trees; we will use the following uniform sub-Gaussian estimate from \cite{addario12tail}. Given a finite graph 
$G$, write $\diam(G)$ for the diameter of $G$. 
\begin{thm} \label{thm:treeheight} \stepcounter{CC}
There exist absolute constants $\cc\ccdef\ccheight,\CC\CCdef\CCheight>0$ such that for all $n \ge 1$, if $T^{(n)}$ is a uniformly random tree 
on labelled vertices $\{1,\ldots,n\}$ then for all $x>0$,
\[
\p{\diam(T^{(n)})\ge x\sqrt{n}} \le \CCx e^{-\ccx x^2}.
\]
\end{thm}
The random variable $Z_i$ is distributed as $B_{Z_i}$, so typically has size of order $(Z_i-1)^{-2}$, and the tree $P_i$ should 
therefore have diameter of order $(Z_i-1)^{-1}$. The next proposition essentially states that the sum of the diameters of the trees $P_1,\ldots,P_i$ is unlikely to be much larger than the diameter of the final tree $P_i$. 
\begin{prop}\label{prop:diamsum}
There exist constants $\cc\ccdef\ccdiam,\CC\CCdef\CCdiam > 0$ such that for all $x > 1$ and all $r > 1$, 
\[
\p{\sum_{i: X_i > 1+1/r} \diam(P_i) \ge x r} \le \CCdiam \log r \cdot e^{-\ccdiam x^{1/2}}\, .
\]
\end{prop}
\begin{proof}
Let $\eps = 1/r$, and recall the definition of the random variable $H_{\eps}$ from Section~\ref{sec:halving}. 
By \refL{lem:halvingtimes}, for $j > 1$ we have 
\[
\p{H_{\eps} > j} < \log_2(2/\eps) e^{-\ccht j} = \log_2 (2r) e^{-\ccht j}\, .
\]
Set $\hat{k} = \ceil{\log_2 (\theta(x_1)/\eps)}$, 
where $x_1$ is as in \refL{lem:halvingtimes}. 
For $0 \le k < \hat{k}\lfloor \log_2 (\theta(x_1)/\eps) \rfloor$ let $\ell_j = I(1+\theta(x_1)/2^j)$, 
and let $\ell_{\hat{k}} = I(1+\eps)$. 

Note that for for $k \le \hat{k}$, if $H_{\eps} \le j$ then $\ell_{k}-\ell_{k-1} \le j$. 
Since $Z_i$ is distributed as $B_{X_i}$, it follows from this observation, a union bound, and 
the upper bound in \refL{lem:bybounds} that for $1 \le k \le \hat{k}$, for all $c > 1$ and $j > 1$, 
\begin{align*}
	& \p{\max\{Z_i: \ell_{k-1}\le i < \ell_k\} \ge c \cdot 2^{2k}/\theta(x_1)^2, H_{\eps} \le j} \\
\le 	& j \sup_{\theta(x_1)/2^{k} < x \le \theta(x_1)/2^{k-1}} \p{ B_x > c \cdot 2^{2k}/\theta(x_1)^2} \\
 \le	& j \sup_{\theta(x_1)/2^{k} < x \le \theta(x_1)/2^{k-1}} \p{ B_x > c/(x-1)^2} \\
\le 	&  20je^{-c/8}\, .
\end{align*}
Writing $c = c'\cdot 2^{\hat{k}-k}$, we have $c \cdot 2^{2k}/\theta(x_1)^2 \le 4 c' \cdot 2^{-(\hat{k}-k)}r^2$, and so 
\[
\p{\max\{Z_i: \ell_{k-1}\le i < \ell_k\} \ge 4c' \cdot 2^{-(\hat{k}-k)}r^2, H_{\eps} \le j} \le 20j e^{-c'/8} e^{-2^{(\hat{k}-k)}/8}\, .
\]
Write $E_k$ for the event whose probability is bounded in the preceding inequality. 
On $E_k$, each of the trees $P_i$, for $\ell_{k-1} \le i < \ell_k$ has size at most $s:=\floor{4c' \cdot 2^{-(\hat{k}-k)}r^2}$. 
Letting $T^{(s)}$ be a uniformly random labeled tree with $s$ vertices, 
by \refT{thm:treeheight} it follows that 
\begin{align*}
& \p{\max\{\diam(P_i):\ell_{k-1}\le i < \ell_k\} \ge 2 c' \cdot 2^{-(\hat{k}-k)/2} r, E_k} \\
\le & j \cdot \p{\diam(T^{(s)}) > \sqrt{c'}\sqrt{s}} \le \CCheight j e^{-\ccheight \cdot c'}\, .
\end{align*}
On the other hand, if $H_{\eps} \le j$ and $\max\{\diam(P_i):\ell_{k-1}\le i < \ell_k\} < 2 c' \cdot 2^{-(\hat{k}-k)/2} r$ for each $k\le \hat{k}$ then 
\[
\sum_{i: X_i > 1+1/r} \diam(P_i) = 
\sum_{1 \le i < \ell_{\hat{k}}} \diam(P_i) < \frac{2 jc'}{1-2^{-1/2}} r < 8 j c' r\, .
\]
It follows that 
\begin{align*}
	&\quad \p{\sum_{i: X_i > 1+1/r} \diam(P_i) \ge 8 c' j r} \\
\le 	&\quad \p{H_{\eps} > j} + 
		\sum_{1 \le k < \hat{k}} \p{\max\{Z_i: \ell_{k-1}\le i < \ell_k\} \ge 4c' \cdot 2^{-(\hat{k}-k)}r^2, H_{\eps} \le j} \\
+	& 	\sum_{1 \le k < \hat{k}} \p{\max\{\diam(P_i):\ell_{k-1}\le i < \ell_k\} \ge 2c' \cdot 2^{-(\hat{k}-k)/2} r, E_k} \\
\le	&\quad  \log_2 (2 r) e^{-\ccht j} + 20j e^{-c'/8} \sum_{i \ge 1} e^{-2^i/8} + \CCheight j e^{-\ccheight \cdot c'}\\
<	&\quad  \log_2 (2 r) e^{-\ccht j} + 40j e^{-c'/8} + \CCheight j e^{-\ccheight \cdot c'}\, .
\end{align*}
Taking $j = \beta c'$ for some small $\beta > 0$ completes the proof. 
\end{proof}

\subsection{The lower bound from \refT{thm:tvolume}.}
For the remainder of Section~\ref{sec:volgor}, for $r > 1$ we write $i_0=i_0(r) = I(1+1/r)$. 
The key to the lower bound is the following proposition, which gives stretched exponential bounds for the lower tail of $|B_{\rT}(\emptyset,r)|$. 

\begin{prop}\label{prop:vglower}
There exist constants $\cc\ccdef\ccglow,\CC\CCdef\CCglow > 0$ such that for all $r > 1$ and all $x > 1$, 
\[
\p{|B_{\rT}(\emptyset,r)| < r^2/x} < \CCx \log r \cdot e^{-\ccx x^{1/8}}\, .
\]
\end{prop}
\begin{proof}
Fix $k > 1$ and let $i_1\ge i_0=I(1+1/r)$ be minimal so that $Z_{i_0} \ge r^2/k$. By the lower tail bound in \refL{lem:bybounds} we have 
\[
\p{i_1 \ge i_0+j} \le (8/k^{1/2})^j\, ,
\]
for all $j \ge 0$. Furthermore, by \refT{thm:treeheight} and a union bound, for $x > 1$ we have 
\[
\p{i_1 < i_0+j,\sum_{i_0 \le i < i_1} \diam(P_i) > xjr/k^{1/2}} \le j \cdot \CCheight e^{-\ccheight x^2}\, .
\]
A result of Luczak and Winkler \cite{luczak04building} implies that uniformly random rooted labelled trees are stochastically increasing. In other words, given $1 \le m \le n$, it is possible to construct a pair $(t_m,t_n)$ such that $t_m$ and $t_n$ are uniformly random labelled trees on $\{1,\ldots,m\}$ and on $\{1,\ldots,n\}$, respectively, and such that $t_m$ is a subtree of $t_n$. This fact implies that, writing $s=\ceil{r^2/k}$, we may find a subtree $T^{(s)}$ of $P_{i_1}$ so that $(T^{(s)},S_{i_1})$ is distributed as a uniformly random rooted tree with $s$ vertices. 
It follows that for all $x > 1$, 
\begin{align*}
\p{|B_{P_{i_1}}(S_{i_1},x r/k^{1/2})| < r^2/k} & \le \p{|B_{T}(S_{i_1},x r/k^{1/2})| < r^2/k} \\
			& = \p{\diam(T^{(s)}) > x r/k^{1/2}} \\
			& \le \CCheight e^{- \ccheight x^2}
\end{align*}
Finally, if $B_{\rT}(\emptyset,3xr)< r^2/k$, 
then either 
$\sum_{i:X_i > 1+1/r} \diam(P_i) > xr$, 
or 
$\sum_{i_0 \le i < i_1} \diam(P_i) > xr$, 
or
$|B_{P_{i_1}}(S_{i_1},xr)| < r^2/k$. 
By Proposition~\ref{prop:diamsum} and the preceding bounds (the first two applied with $j=k^{1/2}$), we then have 
\[
\p{B_{\rT}(\emptyset,3xr)> r^2/k} \le 
\CCdiam \log r \cdot e^{-\ccdiam x^{1/2}} + \pran{\frac{8}{k^{1/2}}}^{k^{1/2}} + (k^{1/2}+1) \CCheight e^{-\ccheight x^2}.
\]
Taking $k=x^2$ yields that there exist constants $c,C > 0$ such that 
\[
\p{B_{\rT}(\emptyset,3xr)> r^2/x^2} \le C \log r \cdot e^{-cx^{1/2}}\, ,
\]
which completes the proof (take $r'=xr$ so that $r^2/x^2= (r')^2/x^4$). 
\end{proof}
We conclude the section by proving the lower bound from Theorem~\ref{thm:tvolume}. 
\begin{thm}\label{thm:tlower}
For any $\eps > 0$, we have 
\[
\p{\liminf_{r \to \infty} \frac{|B_{\rT}(\emptyset,r)|}{r^2/\log^{8+\eps} r} \ge 1} = 1\, .
\]
\end{thm}
\begin{proof}
Fix any non-decreasing function $x:(1,\infty) \to (1,\infty)$. 
If $|B_{\rT}(\emptyset,r)| < r^2/x(r)$ for arbitrarily large $r$, then we must also have that 
also have that $|B_{\rT}(\emptyset,2^i)| < 4 r^2/x(2^i)$ for infinitely many $i$. 
It follows that 
\[
\p{|B_{\rT}(\emptyset,r)| < r^2/x(r)~\mathrm{i.o.}} \le \p{|B_{\rT}(\emptyset,2^i)| < 4 \cdot 2^{2i}/x(2^i)~\mathrm{i.o.}}\, . 
\]
By Proposition~\ref{prop:vglower} we have 
\[
\sum_{i \ge 1} \p{|B_{\rT}(\emptyset,2^i)| < 4 2^{2i}/x(2^i)} \le \sum_{i \ge 1} \CCglow \log (2^i) \cdot e^{-\ccglow x(2^i)^{1/8}}\, .
\]
Taking $x(r)=\log^{8+\eps} r$, the latter sum converges, and the result follows by Borel-Cantelli. 
\end{proof}

\subsection{The upper bound from \refT{thm:tvolume}.}

Let $(\rT_{\mathrm{IIC}},\emptyset)$ be the $\pgw(1)$ incipient infinite cluster, with root $\emptyset$. In other words, this is a $\pgw(1)$ Galton-Watson tree with root $\emptyset$, conditioned to have infinite size (the existence of such a law was shown by Grimmett \cite{grimmett80random}, and was later extended to non-Poisson branching distributions by Kesten \cite{kesten86subdiffusive}). We shall use Theorem~3 from~\cite{addario12prim}, which provides an explicit coupling showing that $(\rT,\emptyset)$ is stochastically dominated by $(\rT_{\mathrm{IIC}},\emptyset)$. In other words, we may work in a space in which $(\rT_{\mathrm{IIC}},\emptyset)$ is almost surely a rooted subtree of $(\rT,\emptyset)$. 

Next, for $k \ge 1$, let $(\rT^k_{\mathrm{IIC}},\rho^k)$ be a Galton-Watson tree with Binomial$(k,1/k)$ branching distribution and root $\rho^k$, conditioned to be infinite. From the fact that the Binomial$(k,1/k)$ law converges in total variation to the Poisson$(1)$ law, it is easily seen that $(\rT^k_{\mathrm{IIC}},\rho^k)$ converges in the local weak sense to $(\rT_{\mathrm{IIC}},\emptyset)$ as $k \to \infty$. We may therefore work in a space in which $(\rT^\rho_{\mathrm{IIC}},r^k) \convas (\rT_{\mathrm{IIC}},\emptyset)$, or in other words, for all $r \in \N$ there is an almost surely finite $k_r$ such that for all $k \ge k_r$, 
\[
\rT_{\mathrm{IIC}}(r) \simeq \rT^k_{\mathrm{IIC}}(r)\, .
\]
From this fact, together with the stochastic domination of $\rT$ by $\rT_{\mathrm{IIC}}$, it follows that for any $r >0$ and $m > 0$, 
we have 
\begin{align*}
\p{|B_{\rT}((\emptyset,r),\emptyset)| \ge m} 
& \le \p{|B_{\rT_{\mathrm{IIC}}}(\emptyset,r)| \ge m}  \\
& = \lim_{k \to \infty}\p{|B_{\rT^k_{\mathrm{IIC}}}(\rho^k,r)| \ge m}\, .
\end{align*}
We now use a bound of Barlow and Kumagai~\cite{barlow06iic} (Proposition~2.7), which states that there exist constants $c_0,c_1$ such that for all $k \in \N$ and all $\lambda > 0$, 
\[
\p{|B_{\rT^k_{\mathrm{IIC}}}(\rho^k,r)| \ge \lambda r^2} \le c_0e^{-c_1 \lambda}\, .
\]
In fact, in \cite{barlow06iic} the bound is not asserted to be uniform in $k$ but this is easily verified to be a consequence of the proof. It follows that for all $r > 0$ and $\lambda > 0$
\begin{equation}\label{eq:forupper_bd}
\p{|B_{\rT}(\emptyset,r)| \ge \lambda r^2} \le c_0e^{-c_1 \lambda}\, .
\end{equation}
We conclude Section~\ref{sec:volgor} by proving the upper bound from Theorem~\ref{thm:tvolume}. 
\begin{thm}\label{thm:tupper}
There exists $C > 0$ such that for any $\eps > 0$, we have 
\[
\p{\limsup_{r \to \infty} \frac{|B_{\rT}(\emptyset,r)|}{r^2 \log\log r} \le C} = 1\, .
\]
\end{thm}
\begin{proof}
Since for $r$ large and $r \le s \le 2r$ we have $r^2 \log\log r \le s^2 \log\log s <  5 r^2 \log\log r$, 
it suffices to prove that there exists $C > 0$ such that 
\[
\p{|B_{\rT}(\emptyset,2^i)| > C \log i \cdot 2^{2i}~\mathrm{i.o}}=0. 
\]
Taking $C = 2/c_1$, by (\ref{eq:forupper_bd}) we have 
\[
\sum_{i \ge 1} \p{|B_{\rT}(\emptyset,2^i)| > C \log i \cdot 2^{2i}} \le c_0 \cdot \sum_{i \ge 1} e^{-c_1 (C \log i)} < c_0 \sum_{i \ge 1} i^{-2} < \infty\, ,
\]
and the result follows by Borel-Cantelli. 
\end{proof}

\section{ {\bf Proof of Theorem~\ref{thm:main}}} \label{sec:mainproof}
Recall from Section~\ref{sec:futuremax} that for $1 \le k \le n$, $\rM_{n,k}$ is the subtree of $\rM_n$ built by the first $k$ steps of Prim's algorithm on $\rK_n$, started from vertex $v_1(\rK_n)=1$. 

Let $k(n)=\lceil \log^5 n\rceil$. In what follows we always assume $n$ is large enough that $k(n)< n$. By Proposition~\ref{prop:abgw} and Skorohod's representation theorem, we may work in a space in which 
\begin{equation}\label{eq:convas}
\lim_{n \to \infty} \sup\{i \in \N: \rM_{n,i} = \rT_i\} \stackrel{\mathrm{a.s.}}{=} \infty\, ,
\end{equation}
and do so for the remainder of the proof. 

In this section we will write both $v_i = v_i(\rK_n) \in V(\rM_{n,i})$ and $v_i = v_i(\rU) \in V(\rT_{i})$, and likewise write both $e_i = e_i(\rK_n)$ and $e_i = e_i(\rU)$, when 
there is little risk of ambiguity. By the comments of the preceding paragraph, at least for fixed $i$ this is not a major abuse of notation. 

Next, recall the definition of $g_n(j,z)$ from~(\ref{eq:gzdef}) 
and, for $1 \le j \le n-1$ and $z> 1$, let 
\[
d_n(j,z) = \inf\{\ell: j < \ell \le n-1, W_n(e_{\ell}) \ge z\},
\]
\nomenclature[Dnjz]{$d_n(j,z)$}{Equals $\inf\{\ell: j < \ell \le n-1, W_n(e_{\ell}) \ge z\}$.} 
or set $d_n(j,z)=n$ if the preceding infimum is empty. In what follows we write $d_n(z)=d_n(k(n),z)$ and $g_n(z)=g_n(k(n),z)$ for succinctness. 

For $z \ge 0$ and for $1 \le j \le n$, let $\cF_n(j,z)$ be the $\sigma$-algebra induced by $\{\rM_{n,i},1 \le i \le j\}$ and by the indicator $\I{W_n(e_{j}) > z}$. 
\nomenclature[Fnjz]{$\mathcal{F}_n(j,z)$}{$\sigma$-algebra induced by $\{\rM_{n,i},1 \le i \le j\}$ and by the indicator $\I{W_n(e_{j}) > z}$.}
(This leads to a sort of filtration that is commonly encountered in probabilistic combinatorics. Informally, $\cF_n(j,z)$ takes us ``part way through'' step $j+1$ of Prim's algorithm: we reveal whether $e_j$ has weight greater than $z$, but leave the discovery of $e_j$'s endpoints and precise weight for later.) 
Note that while $d_n(z)$ is random, it is a stopping time for the filtration $\{\cF_n(j,z),1 \le j \le n\}$ and so $\cF_{n}(d_n(z),z)$ is a $\sigma$-algebra - see \cite{williams91probability}, A 14.1. Also, $\rM_{n,d_n(z)}$ is measurable with respect to $\cF_n(d_n(z),z)$. 

We now run Kruskal's algorithm starting from the graph consisting of $\rM_{n,d_n(z)}$ together with the MSTs of the components of $K_n^z$ disjoint from $\rM_{n,d_n(z)}$. More precisely, for $\lambda \ge z$, let $F_{n}^{z,\lambda}$ be the subgraph of $M_n$ with vertices $\{1,\ldots,n$ and edges 
\[
\{e \in E(M_n): e \in E(M_{n,d_n(z)})~\mbox{or}~W_n(e) \le \lambda\}. 
\]
\nomenclature[Fnzlambda]{$F_{n}^{z,\lambda}$}{Subgraph of $M_n$ with edges 
$\{e \in E(M_n): e \in E(M_{n,d_n(z)})~\mbox{or}~W_n(e) \le \lambda\}$.}
We define $F_n^{z,\lambda-}$ similarly, but with the requirement that $W_n(e) < \lambda$. 
For $v \not\in V(M_{n,d_n(z)})$ we let $x_n(v) = \inf\{\lambda: v \in F_{n}^{z,\lambda}\}$. 
Write
\[
\cG_{n}(\lambda)=\sigma(\rK_n^t,0 \le t \le \lambda) = \sigma(W_n(e)\I{W_n(e) \le \lambda},e \in E(K_n))\, 
\]
for the $\sigma$-algebra containing all information about the graph process $(\rK_{n}^{t},0 \le t \le \lambda)$, 
and likewise define $\cG_{n}(\lambda-)$. 
We then have that $F_{n}^{z,\lambda}$ and $F_n^{z,\lambda-}$ are measurable with respect to $
\hat{\cF}_{n,z,\lambda}=\sigma(\cF_n(d_n(z),z) \cup \cG_{n}(\lambda))$ 
and $
\hat{\cF}_{n,z,\lambda-}=\sigma(\cF_n(d_n(z),z) \cup \cG_{n}(\lambda-))$, respectively. 

Let $M_n^{z,\lambda}$ be the subtree of $F_{n}^{z,\lambda}$ consisting of all nodes in the same component of $F_{n}^{z,\lambda}$ as $1=v_1$ whose path to $v_1$ in $F_{n}^{z,\lambda}$ contains no node $v_j$ with $g_n(z)<j \le d_n(z)$, and let $\rM_n^{z,\lambda}$ be the associated random RWG. 
\nomenclature[Mnzlambda]{$M_n^{z,\lambda}$}{Subtree of $F_n^{z,\lambda}$ induced by the set of vertices whose path to $v_1$ in $F_{n}^{z,\lambda}$ does not pass through $\{v_j,g_n(z)<j \le d_n(z)\}$.}
Next, recall the definition of $\rM(\lambda)$ from Section~\ref{sec:pwag}. For $1 \le i \le g(z)$ let $M^{z,\lambda}$ be the subtree of $M(\lambda)$ consisting of all nodes whose path to the root $\emptyset=v_1$ of $M(\lambda)$ contains no node $v_j$ with $j > g(z)$, and let $\rM^{z,\lambda}$ be the corresponding random RWG. 
\nomenclature[Mzlambda]{$M^{z,\lambda}$}{Subtree of $M(\lambda)$ induced by nodes whose path to $v_1$ in $M(\lambda)$ does not pass through $\{v_j,j > g(z)\}$.}
(Likewise define $M_{n}^{z,\lambda-},M^{z,\lambda-},\rM_{n}^{z,\lambda-}$, and $\rM^{z,\lambda-}$ in the obvious ways). In what follows we write $\rM_n^z$ and $\rM^z$ for $\rM_n^{z,\infty}$ and $\rM^{z,\infty}$, respectively. 

\begin{lem}\label{lem:tnztotn}
For any $z > 1$ we have $\lim_{\lambda \to \infty} \limsup_{n \to \infty} \p{\rM_n^{z,\lambda} \ne \rM_n^z} = 0$. 
\end{lem}

\begin{lem}\label{lem:tntot}
For any fixed $\lambda \ge z$ we have $\rM_n^{z,\lambda} \convdist \rM^{z,\lambda}$ as $n \to \infty$. 
\end{lem}

Assuming the two lemmas, the proof of Theorem~\ref{thm:main} is easily completed. 
By the definition of $g(z)$, the edge $e_{g(z)} = \{p(v_{g(z)+1}),v_{g(z)+1}\}$ is almost surely the last edge of weight at least $z$ added by invasion percolation on $\rU$. It follows that $e_{g(z)}$ is on the unique infinite path from the root $\emptyset$ in $\rT$, and that for all $i > g(z)+1$, $v_i$ is a descendant of $v_{g(z)+1}$. Furthermore,  $\rT$ is locally finite and $g(z) \to \infty$ as $z \downarrow 1$. It follows that for any fixed $r \in (1,\infty)$. 
We thus have 
\[
\lim_{z \downarrow 1} \p{d'_{\rT}(\emptyset,v_{g(z)}) \ge r} = 1. 
\]
By (\ref{eq:convas}), it follows that 
\[
\lim_{z \downarrow 1} \liminf_{n \to \infty} \p{d'_{\rM_{n,k(n)}}(1,g_n(z)) \ge r} = 1\, .
\]
From these facts, it follows that 
\begin{equation}\label{eq:ballsbound}
\lim_{z \downarrow 1} \p{B'_{\rM}(\emptyset,r) \subset V(T^z)} =1\, \quad \mbox{and} \quad
\lim_{z \downarrow 1} \liminf_{n \to \infty} \p{B_{\rM_n}'(1,r) \subset V(M_n^z)} = 1\, .
\end{equation}
Finally, it was observed in Section~\ref{sec:pwag} that $\rM=\rM(\infty)$ is almost surely locally finite, and by Corollary~\ref{cor:nvdist}, below, we have that $\rM$ is almost surely one-ended. It follows that $\rM^z$ is almost surely finite, and so for any fixed $z > 1$ we have 
\[
\lim_{\lambda \to \infty} \p{\rM^{z,\lambda} \ne \rM^z} = 0, 
\]
which combined with Lemmas~\ref{lem:tnztotn} and~\ref{lem:tntot} yields that 
$\rM_n^z \convdist \rM^z$. Together with (\ref{eq:ballsbound}), this implies that 
\[
\rM'_n(r) \convdist \rM'(r)\, 
\]
(recall from the introduction that for an RWG $\rG$, we write $\rG'(r)$ for the sub-RWG induced 
by the set of nodes at weighted distance at most $r$ from the root). 
Since $r$ was arbitrary, this proves Theorem~\ref{thm:main}. We now turn to the proofs of Lemmas~\ref{lem:tnztotn} and~\ref{lem:tntot}. 
In proving both lemmas, we will 
use the following definition. For $z > 1$ and $v \in V(K_n)$ write 
\[
x_n(v,z) = \begin{cases}
			\max\{W_n(e_j), i \le j < d_n(z)\}	& \mbox{ if } v=v_i, i < d_n(z) \\
			\max\{W_n(e_j), d_n(z) \le j < i\}	& \mbox{ if } v=v_i, i \ge d_n(z)\, .
		\end{cases}
\]
We also recall from Section~\ref{sec:pwag} that for $v \in V(M)$, $x(v)$ is the largest weight of any edge in the unique infinite path in $\rM$ starting from $v$, and that $x(v)=a(v)$ for $v \not \in V(T)$. 

Note that for any $\lambda \ge z$ and any $v \in V(K_n)$, the random variable $x_n(v,z) \I{x_n(v,z) \le \lambda}$ is $\hat{\cF}_{n,z,\lambda}$-measurable. 
Note also that for $j \le g_n(z)$ we have $\max\{W_n(e_j), i \le j < d_n(z)\}=\max\{W_n(e_j), i \le j \le g_n(z)\}$ by the definitions of $g_n(z)$ and of $d_n(z)$. Also, since $g(z)$ is almost surely finite, by (\ref{eq:convas}) we have $g_n(z) \convas g(z)$ and so almost surely, for all $n$ sufficiently large, we have $g_n(z)=g(z)$ and $x_n(v_i,z) = x(v_i)$ for all $1 \le i \le g_n(z)$. Furthermore, for $g_n(z) < i \le d_n(z)$, necessarily $x_n(v_i,z) < z$. 

\begin{proof}[Proof of Lemma~\ref{lem:tnztotn}]
For $z \le \lambda \le \infty$, 
note that the component of $F_n^{z,\lambda-}$ containing $1=v_1$ is precisely $M_{n,d_n(\lambda)}$. 
Indeed, by the definition of $d_n(\lambda)$, the vertices of $M_{n,d_n(\lambda)}$ are precisely those vertices of $M_n$ joined to $M_{n,d_n(z)}$ by a path all of whose edges have weight less than $\lambda$. These are precisely the vertices joined to $M_{n,d_n(z)}$ by Kruskal's algorithm started from $F_{n}^{z,z}$ and stopped at weight $\lambda-$. 

Next, for $z < \lambda < \infty$, suppose that $C$ is a component of $F_{n}^{z,\lambda-}$ disjoint from $M_{n,d_n(\lambda)}$ and that $C$ is joined to $M_{n,d_n(\lambda)}$ at time $\lambda$, by some edge $\{v,w\}$ with $v \in V(M_{n,d_n(\lambda)})$ and $w \in V(C)$. By the symmetry of the model, $v$ is equally likely to be any vertex $v\in V(M_{n,d_n(\lambda)})$ with $x_n(v,z) \le \lambda$ (and can not be any vertex $v$ with $x_n(v,z) > \lambda$). 
But almost surely 
\[
\{v \in V(M_{n,d_n(\lambda)}): x_n(v,z) > \lambda\} = \{v_i, 1 \le i \le g_n(\lambda)\} \subset \{v_i, 1 \le i \le g_n(z)\} = V(M_n^{z,z})\, .
\]
Since $V(M_n^{z,z}) \subset V(M_n^{z,\lambda-})$, this implies that for any $\lambda > z$, the end point in 
$V(M_{n,d_n(\lambda)})$ of a new connection at time $\lambda$ is uniformly distributed over 
\[
\{v_i, g_n(\lambda)<i \le d_n(t)\} \supset V(M_{n,d_n(\lambda)})\setminus V(M_{n}^{z,\lambda-})\, .
\]
Since also $|V(M_n^{z,\infty})|=n$, this immediately yields that 
for all $z \le \lambda < \infty$, 
\begin{equation}\label{eq:fortower}
\Cexp{|V(M_n^{z,\infty})|}{\hat{\cF}_{n,z,\lambda}} \le n \cdot \frac{|V(M_n^{z,\lambda})|}{|V(M_{n,d_n(\lambda)})|} 
= n \cdot \frac{|V(M_n^{z,\lambda})|}{d_n(\lambda))}\, .
\end{equation}
Next, since $g(z)$ is a.s.\ finite and $g_n(z) \convas g(z)$ in the space where (\ref{eq:convas}) holds, it follows that for all $\eps > 0$ there is $N_\eps > 0$ such that for $n$ large, 
\[
\p{g_n(z) \ge N_{\eps}} \le \eps/3\, .
\]
Now fix $\eps > 0$ and $0 < \alpha < \theta(z)/2$ small enough that $1/(1-\alpha) < 1 + \eps^3/(3N_{\eps}^2)$. Next, for any $\lambda >1$ write $A_{n,\lambda} = \{|d_n(\lambda) - n\theta(\lambda)| \le \alpha n|\}$. Reprising the argument for (\ref{eq:gzupper}), for $n$ large enough, if $A_{n,\lambda}$ fails to occur then either $|C_{n,2}(\lambda)| \ge \log^3 n$ or $J_n(\lambda) \ge \log^2 n$ or $|C_{n,1}(\lambda) - n\theta(\lambda)| \ge \alpha n/3$, so for any fixed $\lambda > 1$, for $n$ large, 
\[
\p{A_{n,\lambda}^c} < n^{-99}\, .
\]

For any $\eps \le \alpha$, combining bounds from the last three displayed equations, we obtain that 
\begin{align*}
	& \quad \p{|V(M_n^{z,\infty})| \ge 3N_\eps/\eps^2} \\
\le 	& \quad \p{g_n(z) \ge N_\eps} + \p{A_{n,z}^c} + \frac{\eps^2}{3N_\eps} \E{|V(M_n^{z,\infty})| \I{A_{n,z},g_n(z) \le N_\eps}} \\
< 	& \quad \frac{\eps}{3} + \frac{1}{n^{99}} + \frac{\eps^2}{3N_\eps} \cdot n \cdot  \frac{N_\eps}{\theta(z) - \alpha} \\
<	& \quad \eps\, ,
\end{align*}
for $n$ large. The penultimate inequality follows from (\ref{eq:fortower}) applied with $\lambda = z$ and the tower law (since $M_n^{z,z} =M_{n,g_n(z)}$ by definition).  The final inequality holds since $\theta(z) - \alpha \ge \theta(z)/2 \ge \alpha \ge \eps$. 

Finally, by our choice of $\alpha$, and since $\theta(\lambda) \to 1$ as $\lambda \to \infty$, we may choose $\lambda > z$ sufficiently large that $1/(\theta(\lambda)-\alpha) < 1+ \eps^3/(3N_\eps^2)$. By (\ref{eq:fortower}) we have 
\[
\Cexp{ |V(M_n^{z,\infty})| - |V(M_n^{z,\lambda})|}{\hat{\cF}_{n,z,\lambda}} \le |V(M_n^{z,\lambda})| \cdot\pran{ \frac{n}{d_n(\lambda)} - 1}\, .
\]
On $A_{n,\lambda} \cap \{|V(M_n^{z,\infty})| < 3N_{\eps}/\eps^2\}$ we have 
\[
|V(M_n^{z,\lambda})| \cdot\pran{ \frac{n}{d_n(\lambda)} - 1} \le \frac{3N_\eps}{\eps^2} \pran{\frac{1}{\theta(\lambda)-\alpha} - 1} < \eps\, ,
\]
and so 
\begin{align*}
\p{ \rM_n^{z,\infty} \ne \rM_n^{z,\lambda}} & = \p{ |V(M_n^{z,\infty})| - |V(M_n^{z,\lambda})| \ge 1} \\
							& \le \p{A_{n,\lambda}^c} + \p{|V(M_n^{z,\infty})| \ge 3N_\eps/\eps^2} + \eps \\
							& < 3\eps\, ,
\end{align*}
for $n$ large. As $\eps > 0$ was arbitrary this completes the proof. 
\end{proof}
We now proceed to the proof of Lemma~\ref{lem:tntot}. It would be possible to prove the lemma via an appeal to general theory (e.g. Theorem 4.2.5 of \citet{EK}), but verifying the relevant conditions is no simpler than providing a bare-hands proof, so we prefer the latter. 
\begin{proof}[Proof of Lemma~\ref{lem:tntot}] 
Fix $z > 1$. By (\ref{eq:convas}) and the comments just before the proof of Lemma~\ref{lem:tnztotn}, we may work in a space in which almost surely, for $n$ sufficiently large, we have $g_n(z) = g(z)$, $\rM_{n,g_n(z)} \rT_{g(z)}$, and $x_n(v_i,z) = x(v_i)$ for all $i \le g_n(z)$. We work in such a space throughout the proof. 

We begin by considering the case $\lambda=z$. The forest $F_{n}^{z,z-}$ is just the tree $M_{n,d_n(z)}$ together with the components of $K_n^z$ disjoint from $M_{n,d_n(z)}$. 
Since $W_n(e_{g_n(z)})\ge z$ and $W_n(e_i)< z$ for $g_n(z) < i < d_n(z)$, none of $v_{g_n(z)+2},\ldots,v_{d_n(z)}$ are incident to any of $v_1,\ldots,v_{g_n(z)}$. 
 It follows that almost surely $\rM_n^{z,z}=\rM_{n,g_n(z)}$. Similarly, 
 for $i \le g(z)$, the activation time $x(v_i)$ is at least $z$ and so almost surely $\rM^{z,z} = \rT(g(z))$. 
It follows that almost surely $\rM_n^{z,z}=\rM^{z,z}$ for $n$ large. 

Now let $\lambda_0=z$, and for $j \ge 0$ let 
\[
\lambda_{j+1} = \inf\left\{W(e): e=\{u,y\},u \in V(M^{z,\lambda_j}),y\not\in V(M^{z,\lambda_j})\right\}\, .
\]
The preceding infimum is almost surely finite and attained by a unique edge, which we denote 
$f_{j+1}=\{u_{j+1},y_{j+1}\}$, labelled so that $u_{j+1} \in V(M^{z,\lambda_j}),y_{i+1}\not\in V(M^{z,\lambda_j})$. 
Likewise, for $n \in \N$ let $\lambda_{n,0}=z$, and for $j \ge 0$ let 
\[
\lambda_{n,j+1} = \inf\left\{W(e): e=\{u,y\},u \in V(M_n^{z,\lambda_j}),y\not\in V(M_n^{z,\lambda_j})\right\}\, ,
\]
and let $f_{n,j+1}$ attain the infimum and have endpoints $u_{n,j+1}\in V(M_n^{z,\lambda_j}),y_{n,j+1}\not\in V(M_n^{z,\lambda_j})$. 

We will show that for any fixed non-negative integer $j$, it is possible to couple $\rM_{n}^{z,\lambda_{n,j}}$ and $\rM^{z,\lambda_j}$ so that almost surely, for all $n$ sufficiently large, $\lambda_{n,j}=\lambda_{j}$, and $\rM_{n}^{z,\lambda_{n,j}}$ and $\rM^{z,\lambda_j}$ are isomorphic as RWGs. Since $\rU$ is almost surely locally finite, $\lambda_{j} \to \infty$ almost surely as $j \to \infty$, so such a coupling immediately yields the claimed result. 

For $j=0$, we have already established the claim. 
Now fix $j \ge 0$ for which the claim holds, and work in a space in which 
$\lambda_{n,j}=\lambda_{j}$ and $\rM_{n}^{z,\lambda_{n,j}}=\rM^{z,\lambda_j}$ for $n$ large (we gloss the fact that $\rM_{n}^{z,\lambda_{n,j}}$ and $\rM^{z,\lambda_j}$ are isomorphic rather than identical, for ease of exposition). Note that in such a space, we also have $x(v)=x_n(v,z)$ for all $v \in V(M^{z,\lambda_j})$. 

Conditional on $\lambda_j$, on  $\rM^{z,\lambda_j}$ and on $(x(v), v \in V(M^{z,\lambda_j}))$, let $(E_v, v \in V(M^{z,\lambda_j}))$ be independent Exponential$(1)$ random variables, and for each $v \in V(M^{z,\lambda_j})$ let $E_v^+=\max(x(v),\lambda_j)+E_v$. By the definition of the process $(M(\lambda),\lambda \ge 1)$, under this conditioning, $\lambda_{j+1}$ is distributed as $\min\{ E_v^+: v \in V(M^{z,\lambda_j})\}$. 
Furthermore, additionally conditioning on $\lambda_{j+1}$, we have the following properties: 
\begin{itemize}
\item[(i)] the endpoint $u_{i+1}$ of $f_{i+1}$ within $M^{z,\lambda_j}$ is uniformly distributed among those $v \in V(M^{z,\lambda_j})$ with $x(v) \le \lambda_{j+1}$; 
\item[(ii)] the subtree of $\rM^{z,\lambda_{j+1}}$ that attaches at time $\lambda_{j+1}$ (i.e., containing the vertices $V(M^{z,\lambda_{j+1}}) \setminus V(M^{z,\lambda_j})$) is $\pgw(\lambda_{j+1})$-distributed;
\item[(iii)] we have $\rM^{z,\lambda_{j+1}}=\rM^{z,\lambda_j}$ precisely if the subtree from (ii) is finite, which occurs with probability $1-\theta(\lambda_{j+1})$; and 
\item[(iv)] the edge weights of the subtree from (ii) are independent exponentials conditioned to have value at most $\lambda_{j+1}$. 
\end{itemize}

We next work conditional on $\lambda_{n,j}$, on $\rM_n^{z,\lambda_{n,j}}$ and on $(x_n(v,z), v \in V(M_n^{z,\lambda_{n,j}}))$. Under such conditioning, independently for each $v \in V(M_n^{z,\lambda_{n,j}})$, the smallest weight edge incident to $v$ leaving $M_n^{z,\lambda_{n,j}}$ has weight distributed as 
\[
\max(x_n(v,z),\lambda_j)+\mathrm{Exponential}\pran{\frac{n-1}{n-|V(M_n^{z,\lambda_{n,j}})|}}. 
\]
Now, almost surely $M_n^{z,\lambda_{n,j}}=M^{z,\lambda_{j}}$ for $n$ large, and the latter is almost surely finite, since for any fixed $c > 0$, Exponential$((n-1)/(n-c))\convdist$Exponential$(1)$, it follows that we may couple so that almost surely $\lambda_{n,j+1}=\lambda_{j+1}$ for $n$ sufficiently large. Furthermore, under the current conditioning, the end point $u_{n,j+1}$ of $f_{n,j+1}$ is uniformly distributed among those $v \in V(M_n^{z,\lambda_i})$ with $x(v) \le \lambda_{n,j}$, and it follows from (i) above that for $n$ large we may couple so that $u_{n,j+1}=u_{j+1}$. 

Conditional on $u_{n,j+1}$, the second endpoint $y_{n,j+1}$ of $f_{n,j+1}$ is uniformly distributed over the set 
\[
V(K_n)\setminus \pran{V(M_n^{z,\lambda_{n,j}})\} \cup \{y_{n,i}:0 \le i \le j, u_{n,i}=u_{n,j+1}\}}\, .
\]
This set has size between $n-|V(M_n^{z,\lambda_{n,j}})|-j-1$ and $n-V(M_n^{z,\lambda_{n,j}})$. 
Furthermore, we have $\rM_n^{z,\lambda_{n,j+1}} \ne \rM_n^{z,\lambda_{n,j}}$ precisely if $y_{n,j+1} \not \in \{v_i,i \le d_n(\lambda_{n,j+1})\}$, or in other words, precisely if $y_{n,j+1}$ is not joined by Prim's algorithm before time $d_n(\lambda_{n,j+1})$. To bound this probability, fix any $\alpha > 0$, and define the event $A_{n,\lambda_{j+1}}$ as in the proof of Lemma~\ref{lem:tnztotn}. Since $\lambda_{j+1}$ is almost surely finite, for $n$ sufficiently large we have $\p{A_{n,\lambda_{j+1}}} \le n^{-99}$. Furthermore, since almost surely  $\lambda_{j+1}=\lambda_{n,j+1}$ for $n$ large, conditional on $A_{n,\lambda_{j+1}}$, almost surely for all $n$ sufficiently large we have 
\[
1-\theta(\lambda_{n,j+1})-2\alpha< \frac{n-d_n(\lambda_{n,j+1})}{n-V(M_n^{z,\lambda_{n,j}})} < 
\frac{n-d_n(\lambda_{n,j+1})}{n-V(M_n^{z,\lambda_{n,j}})-j-1} < 
1-\theta(\lambda_{n,j+1})+2\alpha\, .
\]
Since $\alpha> 0$ is was arbitrary, it follows by (iii) that we may couple so that almost surely, for $n$ sufficiently large, $\rM^{z,\lambda_{n,j+1}}_n=\rM^{z,\lambda_{n,j}}_n$ if and only if $M^{z,\lambda_{j+1}}=M_{z,\lambda_j}$. 

Finally, given that $y_{n,j+1} \not \in \{v_i,i \le d_n(\lambda_{n,j})\}$, the vertices in $V(M_n^{z,\lambda_{n,j+1}})\setminus V(M_n^{z,\lambda_{n,j}})$ are precisely those of the component of $K_n^{\lambda_{n,j+1}}$ containing $y_{n,j+1}$.  
For $\alpha > 0$ sufficiently small, conditional on $A_{n,\lambda_j}$,  
since $\lambda_{n,j+1}(1-\lambda_{n,j+1}) < 1$, the restriction of $K_n^{\lambda_{j+1}}$ to the complement of $\{v_i,i \le d_n(\lambda_{n,j})\}$ forms a subcritical random graph. It is then standard that the component 
containing $y_{n,j+1}$ asymptotically dominates a $\pgw(\lambda_{n,j+1}(1-\theta(\lambda_{n,j+1})-2\alpha))$ and is asymptotically dominated by a $\pgw(\lambda_{n,j+1}(1-\theta(\lambda_{n,j+1})+2\alpha))$. Since $\lambda_{n,j+1}^*=\lambda_{n,j+1}(1-\theta(\lambda_{n,j+1}))$, it follows from (ii) that we may couple so that almost surely $M^{z,\lambda_{j+1}}=M_n^{z,\lambda_{n,j+1}}$ for $n$ large. Finally, by the definition of $\rK_n$ and of the trees $\rM_n^{z,\lambda}$, the edge weights of the new subtree in $\rM_n^{z,\lambda_{n,j+1}}$ are independent exponentials conditioned to have value at most $\lambda_{n,j+1}$, which together with (iv) immediately allows us to extend the coupling to $\rM^{z,\lambda_{j+1}}$ and $\rM_n^{z,\lambda_{n,j+1}}$. This completes the proof. 
\end{proof}

\section{ {\bf Volume growth in $\rM$: a proof of Theorem~\ref{thm:volume}}}\label{sec:volgorm}
\subsection{The upper bound from Theorem~\ref{thm:volume}.}
Recall that by our construction of $\rM$ from $\rT$, each vertex $u \in V(M)$ has a {\em start time} $x(u)$, which is the largest weight on the unique infinite path in $M$ leaving $u$. 
The removal of all edges of $T$ separates $M$ into a forest containing infinitely many trees. Each such tree is naturally rooted at some vertex $v \in V(T)$: we denote this tree $\cM_v$, and write $N_v=|V(\cM_v)|$ for its size. Also, for $\nu > 1$ we write $\cM_v(\nu)$ for the subtree of $\cM_v$ induced by those nodes $w$ with $x(w) \le \nu$, and write $N_v(\nu)$ for the size of this subtree. In particular, we have $\cM_v(\infty)=\cM_v$. 

Now, given $\nu > \lambda > 1$ and an integer $k \ge 1$, write 
\[
n_k(\lambda,\nu) = \int_{\lambda < x_1 < \ldots < x_k < \nu} \prod_{i=1}^k \frac{1-\theta(x_i)}{1-x_i^*}~\d x_1 \ldots \d x_k\, ,
\]
and set $n_0(\lambda,\nu)=1$. 
\begin{prop}\label{prop:nvdist}
Fix $v \in V(U)$ and $\lambda > 1$. Then 
for any $\nu \in [\lambda,\infty]$ we have 
\[
\Cexp{N_v(\nu)}{v \in V(T),x(v)=\lambda}= \sum_{k \ge 0} n_k(\lambda,\nu) \, .
\]
\end{prop}
Before proceeding to the proof, we note the following corollary.
\begin{cor}\label{cor:nvdist}
$\rM$ is almost surely one-ended. 
\end{cor}
\begin{proof}
Applying the proposition with $\nu=\infty$ we have 
\[
\Cexp{N_v}{v \in V(T),x(v)=\lambda}= \sum_{k \ge 0} n_k(\lambda,\infty)\, ,
\]
which is finite by Proposition~\ref{prop:nkgrowthbd}, below. Since $T$ is a subtree of $U$, and the latter has countably many nodes, it follows that $N_v$ is almost surely finite for all $v \in V(T)$. Since $T$ is one-ended, the corollary follows. 
\end{proof}
\begin{proof}[Proof of Proposition~\ref{prop:nvdist}]
Given $u \in V(M)$, for $k \ge 0$ we say that $u$ has {\em level k} in $M$ if on the shortest path from $u$ to $T$ there are $k+1$ distinct activation times. In other words, level zero nodes are nodes of $T$, level one nodes belong to trees that attach directly to $T$ in the Poisson Galton-Watson aggregation process, and so on. We write $\cM_v^k(\nu)$ for the nodes in $\cM_v(\nu)$ with level $k$ and write $N_v^k(\nu)$ for the number of such nodes. We claim that for all $u \in V(U)$ and all $k \ge 0$ we have 
\begin{equation}\label{eq:nvdist_toprove}
\Cexp{N_v^k(\nu)}{v \in V(T),x(v)=\lambda}=n_k(\lambda,\nu)\, ,
\end{equation}
from which the Proposition immediately follows. The case $k=0$ of (\ref{eq:nvdist_toprove}) is trival. By the definition of $\rM$, the arrival times of connections to $v$ form a Poisson process with rate $(1-\theta(t))$. Furthermore, when a tree attaches at time $t$, it has distribution $\pgw(t^*)$ and so its expected size is $1/(1-t^*)$. It follows that given that $v \in V(T)$ and $x(v)=\lambda$, 
\[
\Cexp{N^1_v(\nu)}{v \in V(T),x(v)=\lambda} = \int_{\lambda}^\nu \frac{(1-\theta(t))}{1-t^*}~\d t\, ,
\]
which handles the case $k=1$. Next, fix $k \ge 1$ and a node $w \in V(U)$. 
Again by the definition of $\rM$, for any $x_k \in (\lambda,\nu)$, we have 
\begin{align*}
\Cexp{\left|\{u \in \cM_v^{k+1}(\nu)|: w~\mbox{an ancestor of}~u\}\right|}{w \in \cM_v^k(\nu),x(w)=x_k} &= \int_{x_k}^\nu \frac{(1-\theta(t))}{1-t^*}~\d t\,  \\
& = n_1(x_k,\nu)\, .
\end{align*}
By induction, the conditional density of nodes in $\cM_v^k(\nu)$ with $x(u)=x$, given that $v \in V(T)$ and $x(v)=\lambda$, is $\frac{\d}{\d x} n_k(\lambda,x)$. We thus have 
\[
\Cexp{N_v^{k+1}(\nu)}{v \in V(T),x(v)=\lambda} = \int_{\lambda}^\nu n_1(x,\nu) \cdot \frac{\d}{\d x} n_k(\lambda,x) \d x
= n_{k+1}(\lambda,\nu)\, ,
\]
where the final equality follows from the definition of $n_{k+1}(\lambda,\nu)$. This proves (\ref{eq:nvdist_toprove}) by induction and so proves the proposition. 
\end{proof}

We next bound the growth of $n_k(\lambda,\infty)$. 
Notice that since $x_i^* = x_i(1-\theta(x_i))$ we may re-express $n_k(\lambda,\nu)$ as 
\[
n_k(\lambda,\nu) = \int_{\lambda < x_1 < \ldots < x_k < \nu} \prod_{i=1}^k \frac{x_i^*}{x_i(1-x_i^*)}~\d x_1 \ldots \d x_k\, .
\]
Since $x_ie^{-x_i}=x_i^*e^{-x_i^*}$ we may again re-express $n_k(\lambda,\nu)$, as 
\[
n_k(\lambda,\nu) = \int_{\lambda < x_1 < \ldots < x_k < \nu} \prod_{i=1}^k \frac{e^{-(x_i-x_i^*)}}{(1-x_i^*)}\, .
\]
\begin{prop}\label{prop:nkgrowthbd}
There exist constants $c,C>0$ such that for all $\lambda > 1$ with $\lambda - 1$ sufficiently small, 
\[
\frac{c}{(\lambda-1)\log(1/(\lambda-1))} \le \sum_{k \ge 0} n_k(\lambda,\infty) \le \frac{e^{C (\log(1/(\lambda-1)))^{1/2}}}{\lambda-1}\, .
\]
\end{prop}
\begin{proof}
First, for any fixed $\nu > \lambda$, we may rewrite the sum under consideration as 
\[
\sum_{k \ge 0} n_k(\lambda,\nu) \cdot \sum_{\ell \ge 0} n_{\ell}(\nu,\infty), 
\]
which will be useful in what follows. 
We begin by proving an upper bound. Since $x^*$ decreases as $x$ increases, for $k \ge 1$ we have 
\begin{align*}
n_k(\nu,\infty) 
		& \le \frac{1}{k!}\frac{1}{(1-\nu^*)^{k}} \int_{(x_1,\ldots,x_k) \in (\nu,\infty)^k}  \prod_{i=1}^k e^{-(x_i-x_i^*)}~\d x_1 \ldots \d x_k\, . \\
		& \le \frac{1}{k!}\frac{1}{(1-\nu^*)^{k}} e^{-k (\nu-\nu^*)}\, . 
\end{align*}
Since $n_0(\nu,\infty)=1$ for all $\nu$, we thus have 
\begin{equation}\label{eq:tailupper}
1 \le \sum_{k \ge 0} n_k(\nu,\infty) \le \exp\pran{\frac{e^{-(\nu-\nu^*)}}{1-\nu^*}}\, .
\end{equation}
Next, recall that $(1+\eps)^*=1-\eps+O(\eps^2)$ as $\eps \downarrow 0$. 
It follows that as $\nu \downarrow 1$, we have 
\begin{align*}
n_k(\lambda,\nu)	& = \int_{\lambda < x_1 < \ldots < x_k < \nu} \prod_{i=1}^k \pran{ e^{-2(x_i-1)+O((x_i-1)^2)}\frac{1}{x_i-1+O((x_i-1)^2)} }~\d x_1 \ldots \d x_k\, , \\
	& = \pran{ 1+ O(\nu-1)}^k \int_{\lambda < x_1 < \ldots < x_k < \nu} \prod_{i=1}^k \pran{\frac{1}{x_i-1} }~\d x_1 \ldots \d x_k  \\
& = \frac{\pran{ 1+ O(\nu-1)}^k}{k!} \pran{\ln \pran{\frac{\nu-1}{\lambda-1}}}^k\, ,
\end{align*}
where the constant implicit in the notation $O(\nu-1)$ may be chosen uniformly over $\nu \in (1,\nu_0)$ for any fixed $\nu_0 > 1$, and uniformly in $k$ and in $\lambda \in (1,\nu)$. 
We thus have 
\[
\sum_{k \ge 0} n_k(\lambda,\nu) = \exp\pran{ (1+O(\nu-1))\ln\pran{\frac{\nu-1}{\lambda-1}}} = \pran{\frac{\nu-1}{\lambda-1}}^{1+O(\nu-1)}\, .
\]
Combined with (\ref{eq:tailupper}), we then obtain that for fixed $\nu_0 > 1$, for any $1 < \lambda < \nu < \nu_0$, 
\begin{align*}
\sum_{k \ge 0} n_k(\lambda,\infty) & \le \exp\pran{\frac{e^{-(\nu-\nu^*)}}{1-\nu^*}}\cdot \pran{\frac{\nu-1}{\lambda-1}}^{1+O(\nu-1)} \\
\sum_{k \ge 0} n_k(\lambda,\infty) & \ge \pran{\frac{\nu-1}{\lambda-1}}^{1+O(\nu-1)}\, .
\end{align*}
For given $\lambda > 1$ with $\lambda > 1$ small, we may optimize the lower bound (up to constants) by taking $(\nu-1)=(\log(1/(\lambda-1)))^{-1}$. 
A straightforward calculation then yields that there is $c > 0$ such that for all $\lambda > 1$ small enough, 
\[
\sum_{k \ge 0} n_k(\lambda,\infty) \ge \frac{c}{(\lambda-1)\log(1/(\lambda-1))}\, .
\]
The upper bound is optimized by taking $(\nu-1)$ of order $((\log(1/(\lambda-1)))^{-1/2})$, which then yields that there is $C>0$ 
such that for all $\lambda > 1$ small enough, 
\[
\sum_{k \ge 0} n_k(\lambda,\infty) \le e^{C (\log(1/(\lambda-1)))^{1/2}} \cdot \frac{1}{\lambda-1}\, .
\]
This completes the proof. 
\end{proof}

We conclude the section by proving the upper bound from Theorem~\ref{thm:tvolume}. 
In the proof we exploit the description of $\rT$ from Section~\ref{sec:wwlprim}, and invite the reader to recall the relevant definitions. Recall also that for $z > 1$ we write $I(z)=\min\{i:X_i \le z\}$. 

Given $r > 1$ write $E_r$ for the event that either $X_{I(1+1/r)} \le 1+1/(r\log^2 r)$ or $d_{\rT}(\emptyset,R_{I(1+1/r)}) < r/\log^2 r$ or $\sum_{i: X_i >1+ 1/(r\log^2 r)} Z_i > r^2 \log^8 r$. 
By Lemmas~\ref{lem:zisizes} and~\ref{lem:pond-dist-lower}, for all $r$ sufficiently large we have we have 
\[
\p{E_r} \le \frac{\CCdistlb}{\log^{4/3} r} + \CCzi \log (r\log^2 r)\cdot e^{-\cczi \log^2 r}
< \frac{2C_2}{\log^{4/3} r}\, .
\]
By Borel-Cantelli it follows that, 
writing writing $L=\sup\{j:E_{2^j}~\mbox{occurs}\}$, we have that $L$ is almost surely finite. 

Given a node $v \in V(T)$ if $v \in V(P_i)$ then $x(v) = X_i$. 
It follows that 
\[
\E{\sum_{v \in V(P_i)} N_v~|~ P_i,X_i} = \sum_{k \ge 0} n_k(X_i,\infty).  
\]
Also, the $X_i$ are decreasing, and $n_k(\lambda,\infty)$ is decreasing in $\lambda$, from this we obtain 
\begin{align}
& \quad \E{\left.\sum_{i: X_i >1+1/(r\log^2 r)} \sum_{v \in  V(P_i)}N_v \right| ((X_j,P_j),j \ge 1)}\nonumber\\
 = & \quad 
\sum_{i: X_i > 1+1/(r\log^2 r)} \pran{|V(P_i)| \cdot \sum_{k \ge 0} n_k(X_i,\infty)}   \nonumber\\
\le & 
\quad \sum_{k \ge 0} n_k(1+1/(r\log^2 r),\infty) \cdot \sum_{i: X_i > 1/(r\log^2 r)} Z_i \, ,\label{eq:some-bound}
\end{align}
where in the final inequality werecall that $Z_i=|V(P_i)|$. 

For fixed $j>1$, if $j > L$ then by the definition of the event $E_{2^j}$ we have  $B_{\rM}(\emptyset,2^j/j^2) \subset \bigcup_{i: X_i >1+ 1/(j^2 2^j)} V(P_i)$, 
and $\sum_{i: X_i > 1+1/(r\log^2 r)} Z_i \le j^8 2^{2j}$. 
Applying (\ref{eq:some-bound}), it then follows that 
\begin{align}
\E{|B_{\rM}(\emptyset,2^j/j^2)| ~|~ j > L} & 
\le 
j^8 2^{2j} \sum_{k \ge 0} n_k(1+1/(j^22^j)),\infty)\nonumber\\
&\le
j^8 2^{2j}\cdot \frac{e^{C \log^{1/2}(j^22^j)}}{1/j^2 2^j} \nonumber\\
& \le (2^j/j^2)^3 \cdot e^{C' \log^{1/2}(2^j/j^2)}\, ,\label{eq:some-other-bound}
\end{align}
the second-to-last inequality by the upper bound in Proposition~\ref{prop:nkgrowthbd}, 
and the last inequality by a suitable choice of $C'$. 

Finally, if $\limsup_{r\to\infty}|B_{\rM}(\emptyset,r)|/(r^3 e^{3C' \log^{1/2} r}) \ge 1$ then 
for infinitely many $j \in \N$, we must have 
$|B_{\rM}(\emptyset,2^j/j^2)| >  (2^j/j^2)^3 \cdot e^{2C' \log^{1/2} (2^j/j^2)}$. 
On the other hand, for any $\ell \in \N$, by (\ref{eq:some-other-bound}) and the conditional Markov inequality we have 
\begin{align*}
&\quad  \p{|B_{\rM}(\emptyset,2^j/j^2)| >  (2^j/j^2)^3 \cdot e^{2C' \log^{1/2} (2^j/j^2)}~\mbox{for infinitely many}~j \in N} \\
\le &\quad  \p{L > \ell} + \sum_{j > \ell} \p{\left.|B_{\rM}(\emptyset,2^j/j^2)| >  (2^j/j^2)^3 \cdot e^{2C' \log^{1/2} (2^j/j^2)}~\right|~ j > L} \\
\le &\quad   \p{L > \ell} + \sum_{j > \ell} e^{-C' \log^{1/2}(2^j/j^2)}\, ,
\end{align*}
and since $L$ is almost surely finite and the sum is convergent, the latter can be made arbitrarily small by choosing $\ell$ large. It follows that 
\[
\p{\limsup_{r\to\infty}|B_{\rM}(\emptyset,r)|/(r^3 e^{3C' \log^{1/2} r}) \ge 1}\, ,
\]
which establishes the upper bound from Theorem~\ref{thm:volume}. 

It is tempting to try to establish a lower bound in a similar manner, using the lower bound from Proposition~\ref{prop:nkgrowthbd}. However, this proposition only provides information about the expected size of the subtrees $\cM_v$. For our volume growth upper bound we have used total size of each subtree, but for a lower bound information about volume growth within these subtrees would be required.  

In the following subsection, we state, without proof, a proposition by which volume growth lower bounds for $\rT$ can be used to obtain corresponding lower bounds for $\rM$. This proposition, is then immediately used to prove the lower bound from Theorem~\ref{thm:volume}; the proof of the proposition then occupies the remainder of the paper. 

\subsection{A key proposition, relating volume growth bounds for $\rM$ and for $\rT$}
\label{sec:alowerbound} 
Recall the definitions of  $M_n(k)$ and 
of $g_n(j,z)$ from Section~\ref{sec:futuremax}, 
and of $d_n(j,z)$ and of $\cF_n(j,z)$ from Section~\ref{sec:mainproof}. 

Fix $z > 1$, and let $k=k(n)$ satisfy $k(n)\ge \log^5 n$ and $k(n)=o(n)$. 
In what follows, we will write $d=d_n(k(n),z)$ and $g=g_n(k(n),z)$ for succinctness. 
Note that while $d=d_n(k,z)$ is random, it is a stopping time for the filtration $\{\cF_n(j,z),1 \le j \le n\}$ and so $\cF_{n}(d,z)$ is a $\sigma$-algebra - see \cite{williams91probability}, A 14.1.

The key to our lower bound is the following estimate. 
Let $H_n=H_n(k(n),z)$ be the forest obtained from $M_n$ by removing the edges of $M_n(d)$, so $H_n$ has edges $e_d,\ldots,e_{n-1}$. Note that this forest consists of $d$ connected components (trees), which we view as rooted at $v_1,\ldots,v_{d}$. For $1 \le j \le d$, we write $U_{n,j}=U_{n,j}(k(n),z)$ for the vertex set of the component of $H_n$ rooted at $v_j$, and for $r \ge 1$ write $U_{n,j}^{r}$ for the set of vertices of $U_{n,j}$ whose distance to $v_j$ (in $H_n$) is at most $r$.
\begin{prop} \label{prop:lowerboundkey}
There is an absolute constant $M > 1$ such that the following holds. For all $z > 1$ with $z-1$ sufficiently small, 
for any random subset $S$ of $\{v_j, g < j \le d\}$ that is $\cF_{n}(d,z)$-measurable, and any $A > 1$, we have 
\[
\p{\sum_{i \in S} \left|U_{n,i}^{M/(z-1)}\right| \ge \frac{A}{\theta(z)^3}}
\ge \pran{\p{|S| \ge \frac{3A}{\theta(z)^2}} - o_n(1)}\pran{1-\frac{3M}{A} - 4\theta(z)}\, .
\]
\end{prop}
Before proving this proposition, we use it to complete the proof of the lower bound from Theorem~\ref{thm:volume}
\subsubsection{The lower bound from Theorem~\ref{thm:volume}}
Fix $r > 1$ and let $z=z(r)=1+(\log^{13} r)/r$. With $k(n)$ as above, continue to write  $g=g_n(k(n),z)$ and $d=d_n(k(n),z)$, and let 
\[
S = S(n,r) = B_{\rM_n(d)}(1,r) \setminus \{v_1,\ldots,v_g\}. 
\]
By definition, $S \subset \{v_{g+1},\ldots,v_d\}$ and $S$ is an $\cF_n(d,z)$-measurable set. To use Proposition~\ref{prop:lowerboundkey}, we need probability bounds on the lower tail of $|S|$. 

Fix any function $f(n)=o(n)$ with $f(n) \ge \log^5 n$. By Proposition~\ref{prop:abgw} we have 
\[
\rM_n(f(n) \convdist \rT
\]
in the local weak sense. Furthermore, by (\ref{eq:giantldbound}) and (\ref{eq:gzupper}) we have $\p{f(n) \le d} \ge 1-O(n^{-99})$, so for any $x > 1$, by Proposition~\ref{prop:vglower} we have 
\begin{align*}
\p{B_{\rM_n(d)}(1,r) < r^2/x} & \le \p{B_{\rM_n(f(n))}(1,r) < r^2/x} + O(n^{-99}) \\
					& \le \p{B_{\rT}(\emptyset,r) < r^2/x} + o_n(1) \\
					& < \CCglow \log r e^{-\ccglow x^{1/8}} + o_n(1)\, .
\end{align*}
Next, recall the definition of the forward maximal process $((X_i,Z_i),i \ge 1)$ and of the subtrees $P_i$ of $T$ from Section~\ref{sec:wwlprim}. 
Write $i = i(\rT,z) = \sup\{j: X_j \ge z-1\}$, and let $\rT^{(z)}$ be the sub-RWG of $\rT$ induced by the vertices in $P_1,\ldots,P_i$, so $\rT^{(z)}$ has $\sum_{j=1}^i |P_j| = \sum_{j=1}^i Z_j$ vertices. 
By Proposition~\ref{prop:abgw} and and (\ref{eq:gzupper}), we have $\rM_n(g) \convdist \rT^{(z)}$, so by Lemma~\ref{lem:zisizes}, for any $y > 1$, 
\begin{align*}
\p{g \ge \frac{yr^2}{\log^{13} r}} & = \p{g \ge \frac{y}{(z-1)^2}} \\
& \le \CCzi \log ((z-1)^{-1})e^{-\cczi y^{1/2}} 
+o_n(1) \\
& \le \CCzi  \log r e^{-\cczi y^{1/2}} + o_n(1)
\end{align*}
Taking $x=\log^{9} r$ and $y=\log^3 r$, for $r$ sufficiently large we have 
$r^2/x - yr^2/\log^{13} r \ge r^2/(2 \log^{9} r)$ and  $\frac{\CCzi}{2} \log r e^{-\cczi y^{1/2}} \le \CCglow \log r e^{-\ccglow x^{1/8}}$, so 
\[
\p{|S| \le \frac{r^2}{2 \log^{9} r}} \le 2 \CCglow \log r e^{-\ccglow x^{1/8}} + o_n(1)
\le  \frac{1}{r^2} + o_n(1)\, ,
\]
the last inequality holding for $r$ sufficiently large. 
Write $A = r^2 \theta(z)^2/(6\log^9 r)$. For $r$ large we have $\theta(z) \sim 2 (\log^{13} r)/r$, so $A \sim (\log^{2\cdot 13-9} r)/6$. It then follows from Proposition~\ref{prop:lowerboundkey} that for $r$ large, 
\begin{align*}
\p{\sum_{i \in S} \left|U_{n,i}^{M/(z-1)}\right| \ge \frac{A}{\theta(z)^3}} 
& \ge \pran{1-\frac{1}{r^2} - o_n(1)}\pran{1 - \frac{3M}{A} - 4\theta(z)} \\
& > 1 - \frac{19 M}{\log^{17} r} - o_n(1)\, .
\end{align*}
For $r$ large we have $A/\theta(z)^3 \ge r^3/(50 \log^{13+9} r)$, and also have   
and $M/(z-1) < r$ so $\bigcup_{i\in S} U_{n,i}^{M/(z-1)} \subset B_{\rM_n}(1,2r)$. 
It follows that 
\[
\p{|B_{\rM_n}(1,2r)| \ge \frac{r^3}{50 \log^{22} r}} > 1 - \frac{19 M}{\log^{17} r} - o_n(1)\, .
\]
By Theorem~\ref{thm:main}, the latter bound implies that for all $r$ sufficiently large, 
\[
\p{|B_{\rM}(0,2r) \ge  \frac{r^3}{50 \log^{22} r}} > 1 - \frac{19 M}{\log^{17} r}\, .
\]
Now write $r_i = e^i$ for $i \ge 1$.
Then $\sum_{i\ge 1} \log^{17} r_i  = \sum_{i \ge 1} i^{17} < \infty$ and it follows by Borel-Cantelli that 
\[
\p{|B_{\rM}(0,r_i)| <  \frac{r_i^{3}}{400 \log^{22} r_i}\mbox{ for infinitely many i}} = 0. 
\]
Finally, for $r \in (e^i,e^{i+1})$, if $|B_{\rM}(0,r_i)| \ge  \frac{r_i^{3}}{400 \log^{22} r_i}$ then 
$|B_{\rM}(0,r)| \ge \frac{r^3}{400e^3 \log^{22} r}$, so 
\[
\p{\liminf_{r \to \infty} \pran{|B_{\rM}(0,r)| \cdot  \frac{\log^{22} r_i}{r_i^{3}}}>0} = 1\, ,
\]
proving the lower bound from Theorem~\ref{thm:volume}. We now turn to the proof of Proposition~\ref{prop:lowerboundkey}, which is at the heart of the lower bound. 

\subsection{A heuristic argument for Proposition~\ref{prop:lowerboundkey}}

Fix $\eps > 0$ with $(1+\eps)\theta(z) < 1$, and let $E_{n,z}$ be the event that 
\[
1-\eps < \frac{d-g}{\theta(z)\cdot n} \le \frac{d}{\theta(z)\cdot n} < 1+\eps\, .
\]
By (\ref{eq:giantldbound}) and (\ref{eq:gzupper}), this event occurs with probability $1-O(n^{-99})$. We write $\Ehat{\cdot}$ as shorthand for the conditional expectation 
\[
\E{\ \cdot\ \I{E_{n,z}}~|~\cF_{n}(d,z)} = \E{\ \cdot ~|~\cF_{n}(d,z)} \I{E_{n,z}}\, ,
\]
and likewise write 
\[
\phat{\cdot} = \p{\cdot ,E_{n,z}~|~\cF_{n}(d,z)} = \p{\cdot ~|~\cF_{n}(d,z)} \I{E_{n,z}}\, ,
\]
each of the second equations holding since $E_{n,z} \in \cF_{n}(d,z)$. 

We now work conditional on $\cF_n(d,z)$. Using the notation $a \vee b = \max(a,b)$, for $1 \le i \le j \le n-1$  write
\[
w_n^+(i,j)=z\I{W_n(e_j)\ge z}\vee \max\{W_n(e_l), i \le l < j\}\, .
\]
 Then for $1 \le j \le d$, $\phat{v \in U_{n,j}}$ is a measurable function of the forward maximal weights $\{w_n^+(j,d),1 \le j \le d\}$. Note that on $E_{n,z}$ we have $d < n$, so $W_n(e_d) \ge z$ and thus $w^+(j,d) \ge z$ for all $1 \le j \le d$. Furthermore, $w^+(j,d) > z$ for $1 \le j \le g$. 
Also, $w_n^+(j,d)$ is decreasing in $1 \le j \le d$, and for $1 \le i < j \le d$ we have 
\[
\phat{v \in U_{n,u}} \leas \phat{v \in U_{n,j}}\, ,
\]
with almost sure equality holding on the event that $w_n^+(i,d)=w_n^+(j,d)$. 
In particular, since $w_n^+(j,d)=z$ for $g < j \le d$, for such $j$ we have 
\begin{align}
\phat{v \in U_{n,j}} & \geas \frac{1}{d} \I{E_{n,z}} \ge \frac{1}{(1+\eps) \theta(z) \cdot n}\I{E_{n,z}}\, , \label{stochasticordering}\\
\phat{v \in U_{n,j}} & \leas \frac{1}{d-g} \I{E_{n,z}} \le \frac{1}{(1-\eps) \theta(z) \cdot n}\I{E_{n,z}}\, . \nonumber\\
\end{align}
For $z$ close to $1$, say $z=1+\gamma$ for $\gamma > 0$ small, $\theta(z)$ is near $2\gamma$. Since, for fixed $\eps$,  the probability $\p{E_{n,z}^c}$ decays exponentially in $n$, it follows straightforwardly that 
for any fixed vertex $w$, 
\begin{equation}\label{eq:suggestive}
\E{|U_{n,j}|~|~w=v_j, g < j \le d} = (1+o_{\gamma \downarrow 0}(1)) \frac{1-2\gamma}{2\gamma} 
= \frac{1+o_{\gamma \downarrow 0}(1)}{2\gamma}\, .
\end{equation}
Together with the weak convergence result Proposition~\ref{prop:abgw}, the results of Proposition~\ref{prop:vglower} and of Lemma~\ref{lem:zisizes} suggest that $g$ is around $\gamma^{-2}$. 
Similarly, Theorems~\ref{thm:tlower} and \ref{thm:tupper} suggest that typically, a substantial fraction of the nodes in $M_n(g)$ have distance around $1/\gamma$ from the root $v_1$. If the trees $\{U_{n,w},w \in V(M_n(g))\}$ were typically of size $1/\gamma$ (which is plausible in light of (\ref{eq:suggestive})), and additionally were typically of diameter $O(1/\gamma)$, we would then obtain around $|V(M_n(g))|/\gamma \approx 1/\gamma^3$ nodes within distance $1/\gamma$ from the root $v_1$. 

There are two problems with this heuristic argument. First, (\ref{eq:suggestive}) does not apply to nodes of $M_n(g)$. Indeed, if $v_j \in V(M_n(g))$ and $w_n^+(v_j,d)=z' > z$ then the conditional expected size of $U_{n,v_j}$ is around $1/(z'-1)$, which may be much smaller than $\gamma$. We address this problem by instead considering a suitable collection of around $1/\gamma^2$ nodes of $\rM_n(d)$ that are not in $\rM_n(g)$, but that were added in the early stages of Prim's algorithm (shortly after $\rM_n(g)$ was built) and that also have distance around $1/\gamma$ from $v_1$. 

Second, and more importantly, the (identically distributed) random variables $|U_{n,v_j}|$, $g < j \le d$ are not concentrated; their distribution is asymptotically that of $|\pgw(z^*)|$, where $z^*$ is the dual parameter to $z$ and is near $1-\gamma$ for $\gamma=z-1$ small. In particular, (\ref{eq:meanvar}) then says that for such $j$, $\Ehat{U_{n,v_j}^2}$ is around $1/\gamma^3$. The correct picture is not that the $U_{n,v_j}$ are typically of size $1/\gamma$. Rather, the typical size is $O(1)$, but an approximately $\gamma$ proportion of the $\{U_{n,v_j},g<j\le d\}$ have size $1/\gamma^2$, and the latter typically have height of order $1/\gamma$. To capture this picture, and thereby prove a volume growth lower bound, we study the first and second moments of the sizes of a carefully chosen family of subtrees of the trees $\{U_{n,v_j},g<j\le d\}$. We now turn to details. 

\subsection{The proof of Proposition~\ref{prop:lowerboundkey}}
We continue to work conditional on $\cF_{n}(d,z)$. Fix a vertex $u \not\in V(M_n(d))$, and consider the following procedure, which we denote \zpu. The short description of the procedure is this: start a tree-building exploration procedure from $u$. Use Prim's algorithm for edges of weight greater than $z$, and breadth-first search for edges of weight less than $z$; stop the first time a vertex of $M_{n}(d)$ is added. For the sake of clarity, and to introduce some needed notation, we now explain the procedure more carefully. 

List the components of $K_n^z$ that are disjoint from $M_n(d)$ as $\cC=\cC(z)=(\hat{C}_i(z),i \ge 1)$. (There are only a finite number of such components, but we gloss this issue to avoid unnecessary notation.) Edges within the components of $\cC$ have weight at most $z$, whereas edges from these components to one another and to vertices in $M_n(d)$ have weight greater than $z$. 

The vertex $u$ lies in some component from $\cC$. Explore this component via breadth-first search 
(exploring the children of a given node in increasing order of label), and write $D_1(u)$ for the resulting breadth-first search tree. Next, for given $i \ge 1$, suppose that breadth-first search spanning trees $D_1(u),\ldots,D_i(u)$ of some set of components of $K_n^z$ have already been constructed. Add the smallest weight edge $e$ from one of $D_1(u),\ldots,D_i(u)$ to the rest of the graph (this edge has weight greater than $z$).
If the endpoint $v$ of $e$ not in $D_1(u),\ldots,D_i(u)$ is {\em not} a vertex of $M_n(d)$, then 
let $D_{i+1}(u)$ be the breadth-first search spanning tree of the component from $\cC$ containing $v$, and continue exploring. 
If $v$ {\em does} lie within $M_n(d)$, then write $\alpha(u)=v$ and stop. 

Write $\cT(u)$ for the tree built by \zpu, and write $m(u)$ for the number of components of $K_n^z$ explored by \zpu before it stops; so, $\cT(u)$ is composed of $D_1(u),\ldots,D_{m(u)}(u)$, plus the single vertex $\alpha(u)$, which is the 
unique vertex belonging to both $\cT(u)$ and to $M_n(d)$. 
Note that if the components of $K_n^z$ spanned by $D_1(u),\ldots,D_{m(u)}(u)$ happen to be trees (we will shortly see that this occurs whp), then $\cT(u)$ is precisely the restriction of $M_n$ to $V(\cT(u))$. 
Note also that the trees $(\cT(u),u \not\in V(M_n(d)))$ need not be disjoint; for example, if $v \in V(\cT(u))$ then $\cT(u)$ shares at least the vertices $v$ and $\alpha(v)=\alpha(u)$ with $\cT(v)$. 

Let $\tau_u=1+\sum_{i=1}^{m(u)} |D_i(u)|$, and let $\lambda_u=m(u)+\sum_{i=1}^{m(u)} \diam(D_i(u))$, so $\tau_u=|\cT(u)|$ and $\diam(\cT(u)) \le \lambda_u$. Writing $\gamma=z-1$ as before, 
we will next show that $\Ehat{\tau_u}$ and $\Ehat{\lambda_u}$ are of orders $\gamma^{-2}$ and $\gamma^{-1}$, respectively. 

The argument of this paragraph is similar to the one appearing just after the statement of Proposition~\ref{prop:abgw}. 
Each time \zpu adds an edge not lying within a component of $\cC$, the {\em vertex} that is added is equally likely to be any vertex from an unexplored component of $\cC$ or to be any vertex from $\{v_i, g < i \le d\}$. (It also may be a vertex of $M_n(g)$, but this is less likely since, as noted earlier, $\cF_{n}(d,z)$ provides ``stronger lower bounds'' on the weights of edges connecting $M_n(g)$ with the rest of the graph.) Now fix $\ell \ge 1$ and condition that $m(u) \ge \ell$, and let $v$ be the first vertex added by \zpu after fully exploring $D_1(u),\ldots,D_\ell(u)$. Then
\begin{align*}
\phat{v \in V(M_n(d))~|~m(u) \ge \ell,D_1(u),\ldots,D_\ell(u)} 	& \ge \frac{d-g}{n-d-\sum_{i=1}^\ell |D_\ell(u)|} \I{E_{n,z}} \\
											& \ge (1-\eps) \theta(z)  \I{E_{n,z}}.  
\end{align*}
Since the right-hand side does not depend on $D_1(u),\ldots,D_\ell(u)$, by averaging we thus have 
\begin{equation}\label{eq:hittndlower}
\phat{v \in V(M_n(d))~|~m(u) \ge \ell}  \ge (1-\eps) \theta(z)  \I{E_{n,z}}. 
\end{equation}

Given that $m(u) > \ell$ (i.e., that $v \not\in V(M_n(d))$), the graph $K_n^z[\{1,\ldots,n\}\setminus (V(M_{n}(d)) \cup \bigcup_{i=1}^\ell D_i(u))]$ is stochastically dominated by $K_{n-(d-g)}^z$, and $v$ is a uniformly random vertex of this graph. Assuming $z < 3/2$, say, on the event $E_{n,z}$, uniformly in $\eps > 0$ sufficiently small we have 
\begin{align}
(n-(d-g))(1-e^{-z/n}) 	& \le n(1-\theta(z)(1-\eps))\frac{(1+o_n(1))z}{n} \nonumber\\
				& \le (1+o_n(1))  z (1-\theta(z)+\eps) \nonumber\\
				& < z^* + 2\eps\, , \label{eq:gnzduality}
\end{align}
the final inequality holding for $n$ sufficiently large. It follows from standard results about subcritical random graphs (see, e.g., \cite{bollobas01random} Corollary 5.24) that for $n$ sufficiently large, 
\begin{equation}\label{eq:ptreelb}
\phat{D_{\ell+1}(u)~\mbox{is a tree}~|~m(u) > \ell} \ge (1-n^{-1/2}) \I{E_{n,z}}. 
\end{equation}
Also, using (\ref{eq:gnzduality}), it is straightforward to check that for $\eps > 0$ sufficiently small and $n$ sufficiently large, on $E_{n,z}$ we have $\mathrm{Bin}(n-(d-g),(1-e^{-z/n})) \pst \mathrm{Poisson}(z^*+3\eps)$, where $\pst$ denotes stochastic domination.\footnote{For any fixed $x > 0$ and $\eps > 0$, for all $n$ sufficiently large, $\mathrm{Bin}(n,x/n) \preceq_{\mathrm{st}} \mathrm{Poisson}(x+\eps)$.} 
By considering the breadth-first construction of $D_{\ell+1}(u)$, it follows that {\em given} that $m(u) > \ell$ and that $D_{\ell+1}(u)$ is a tree, $D_{\ell+1}(u)$ is stochastically dominated by $\PGW(z^*+3\eps)$.\footnote{To carefully verify this, one should use that conditional on the vertex set of $D_{\ell+1}(u)$, the event that  $D_{\ell+1}(u)$ is a tree is a decreasing event; as this is rather standard and distracts from the flow of the argument, we omit the details.}

Next, (\ref{eq:hittndlower}) implies that $\phat{m(u) > \ell~|~m(u) \ge \ell} \le (1-(1-\eps)\theta(z))\I{E_{n,z}}$, and so on $E_{n,z}$, $m(u)$ is stochastically dominated by a Geometric$((1-\eps)\theta(z))$ random variable. It follows that (still conditional on $\cF_{n}(d,z)$), on the event $E_{n,z}$ we have 
\[
\tau_u \preceq_{\mathrm{st}} \sum_{i=1}^G |P_i|,
\]
where $G$ is Geometric$((1-\eps)\theta(z))$ and, independently of $G$, the $(P_i,i \ge 1)$ are iid $\PGW(z^*+3\eps)$. By Wald's identity and (\ref{eq:meanvar}), we then obtain that 
\begin{equation}\label{eq:tauupper}
\Ehat{\tau_u} \le \I{E_{n,z}}\cdot \frac{1}{(1-\eps)\theta(z)} \cdot \frac{1}{1-z^*-3\eps} \le \I{E_{n,z}}\cdot \frac{1}{\gamma^2}\, ,
\end{equation}
for $\eps > 0$ small enough. (Here we are still writing $\gamma=z-1$, and use that $\theta(z) \sim 2(z-1)$ and that $z-1 \sim 1-z^*$, both as $z \downarrow 1$, in the last inequality.)
Also, by (\ref{eq:ptreelb}) and the stochastic bound on $m(u)$, we easily obtain that 

\begin{align*}
\phat{\mbox{The components of }K_n^z\mbox{ spanned by }D_1(u),\ldots,D_{m(u)}(u)\mbox{ are all trees}} 
& \\
\ge \pran{1-\frac{\theta(z)\cdot\log n}{n^{1/2}}}\cdot \I{E_n,z}\, . &
\end{align*}

We may bound $\Ehat{\lambda_u}$ in a similar fashion to $\Ehat{\tau_u}$; for this we use that for $\eps$ sufficiently small, $z^*+3\eps < 1$. The diameter of a $\PGW(1)$ tree has finite expectation (this is standard, but also follows easily from (\ref{eq:pgwlambdabound}) and \refT{thm:treeheight}), and we obtain 
that on $E_{n,z}$, the random variable $\lambda_u$ is dominated by the sum of a Geometric$((1-\eps)\theta(z))$ number of iid random variables with some finite expectation $F$. It follows that
\begin{equation}\label{eq:lambdaupper}
\Ehat{\lambda_u} \le \I{E_{n,z}} \cdot \frac{F}{(1-\eps)\theta(z)} < \I{E_{n,z}}\cdot \frac{F}{\gamma}\, ,
\end{equation}
for $\eps > 0$ sufficiently small and $z$ sufficiently close to $1$. 

By reprising the above argument, we can obtain a stochastic lower bound on $\tau_u$ that is of roughly the same form. First, given $\ell \le \tau_u$, by {\em the tree built by the first $\ell$ steps of \zpu,} we mean the subtree of $\cT(u)$ consisting of the first $\ell$ vertices added by \zpu. The ``order of addition of vertices'' is well-defined since we explore the subtrees $D_1,\ldots,D_{m(u)}$ in order, and within each subtree the vertices are explored in breadth-first search order. More precisely, we may think of each step of \zpu\ as consisting of the breadth-first-search exploration of a single node, plus possibly the connection of a new component from $\cC$ (the latter occurring each time the current BFS exploration concludes). 

Let $\tau^-_u = \min(\tau_j,\lfloor \eps n\rfloor)$. 
The tree built by the first $\tau^-_u$ steps of \zpu\ is a subtree of\, $\cT(u)$, and is equal to $\cT(u)$ precisely if $\tau_u \le \lfloor \eps n \rfloor$. 
For $\ell < \tau_u^-$, at step $\ell$ of \zpu, some tree $D_i$ is partially built. The number of new nodes added to $D_i$ in the BFS exploration at step $\ell$ has distribution $\mathrm{Bin}(n-d-\ell,1-e^{-z/n})$. 
For $\eps>0$ small, on $E_{n,z}$ we have $n-d-\ell > n(1-2\theta(z))$. Since also $1-e^{-z/n} \ge (1-\eps) z/n$ for $n$ large, on $E_{n,z}$ the number of new nodes thus stochastically dominates
\begin{equation}\label{eq:binstlb}
\mathrm{Bin}\pran{n(1-2\theta(z)),\frac{(1-\eps)z}{n}} \sust \mathrm{Poisson}(z(1-3\theta(z)))\, ,
\end{equation}
the last stochastic inequality again holding for $n$ large. 
Furthermore, for fixed $i \ge 1$, on the event that $m(u) \ge i$ we have 
\begin{align}
\phat{m(u) > i~\left|~m(u) \ge i,\sum_{j=1}^i |D_j| < \lfloor \eps n \rfloor\right.} 
& > \I{E_{n,z}}\cdot \frac{n-(1+\eps) \theta(z) n - \eps n}{n} \nonumber\\
& \ge \I{E_{n,z}} (1-2\theta(z))\, . \label{eq:mulb}
\end{align}
Together, (\ref{eq:binstlb}) and (\ref{eq:mulb}) imply that on $E_{n,z}$, 
\begin{equation}\label{eq:streltauminus}
\tau^-_u \sust \min(\lfloor \eps n\rfloor,\sum_{i=1}^{G^-}|P_i^-|)\, ,
\end{equation}
where $G^-$ is Geometric$(1-2\theta(z))$ and, independent of $G^-$, the $(P_i^-,i \ge 1)$ are iid 
$\pgw(1-2\theta(z))$. 

For $\gamma = z-1$ sufficiently small, $3\gamma < 2\theta(z) < 4\gamma$, so 
\[
\p{G^- > \frac{1}{50\gamma}} > \frac{9}{10}\, .
\]
Also, by (\ref{eq:meanvar}), 
\[
\E{|P_1^-|} \ge \frac{1}{4\gamma},\quad \E{|P_1^-|^2} \le \frac{1}{2\gamma^3}, 
\]
so by a routine application of the Paley-Zygmund inequality, 
\[
\p{\sum_{i=1}^{\lceil 1/50\gamma\rceil} |P_i^-| \ge \frac{1}{400\gamma^2}} \ge \frac{1}{100}\, .
\]
By the independence of $G^-$ and the $P_i^-$, it follows that 
\[
\p{\sum_{i=1}^{G^-} |P_i^-| \ge \frac{1}{400\gamma^2}} > \frac{9}{1000}. 
\]
We also have $\p{\sum_{i=1}^{G^-} |P_i^-| \ge \lfloor \eps n \rfloor} \to 0$ as $n \to \infty$, and so 
by the stochastic relation (\ref{eq:streltauminus}) we obtain 
\[
\phat{\tau^-_u \ge \frac{1}{400\gamma^2}} \ge \I{E_{n,z}} \pran{\frac{9}{1000}-o_n(1)} \ge \I{E_{n,z}} \cdot \frac{1}{125}\, ,
\]
for $n$ sufficiently large. Since $\tau^-_u \le \tau_u$ almost surely, combined with (\ref{eq:tauupper}) and (\ref{eq:lambdaupper}) this implies that there is $C>1$ such that
\[
\phat{\tau_u \in \left[ \frac{1}{C\gamma^2},\frac{C}{\gamma^2}\right],\lambda_u \le \frac{C}{\gamma}} \ge\I{E_{n,z}} \cdot \frac{1}{C}\, ;
\]
this bound holds uniformly over all $z$ with $z-1$ sufficiently small, over all $\eps$ with $0 < \eps < \eps_0(z)$, and for all $n$ greater than some fixed $n_0=n_0(z,\eps)$. 

Fix $v_\ell \in V(M_n(d))$ with $g < \ell \le d$, and let 
\[
N_\ell = \#\left\{u \not\in V(M_n(d)): \alpha(u)=v_\ell, \tau_u \in \left[ \frac{1}{C\gamma^2},\frac{C}{\gamma^2}\right],\lambda_u \le \frac{C}{\gamma}\right\}\, .
\]
For any $u \not\in V(M_{n}(d))$, $\alpha(u) \in \{v_1,\ldots,v_d\}$.  
By (\ref{stochasticordering}) and by symmetry, we have 
\[
\phat{\alpha(u)=v_\ell~|~\tau_u \in \left[ \frac{1}{C\gamma^2},\frac{C}{\gamma^2}\right],\lambda_u \le \frac{C}{\gamma}} \ge \I{E_{n,z}} \cdot \frac{1}{d}\, ,
\]
and we similarly have $\phat{\alpha(u)=v_\ell} \le \I{E_{n,z}}/(d-g) \le \I{E_{n,z}}/((1-\eps)\theta(n))$. 
On $E_{n,z}$, for $\eps > 0$ sufficiently small, $n-d \ge n(1-2\theta(z))$ and $d \le (1+\eps)\theta(z)$, 
so 
\begin{equation}\label{eq:nklower}
\Ehat{N_\ell} \ge \I{E_{n,z}} \cdot \frac{n-d}{d} \ge \I{E_{n,z}}\cdot \frac{1-2\theta(z)}{(1+\eps)\theta(z)}\, .
\end{equation}
Now fix an $\cF_{n}(d,z)$-measurable set $S \subset \{v_{g+1},\ldots,v_d\}$, and write 
$N_S := \sum_{i: v_i \in S} N_i$. Note that $N_{i}$ counts a subset of the vertices in $U_{n,i}^{C/\gamma}=U_{n,i}^{C/(z-1)}$, so to prove Proposition~\ref{prop:lowerboundkey} it suffices to prove that for any $A > 1$, 
\begin{equation}\label{eq:nstoprove}
\p{N_S \ge \frac{A}{\theta(z)^3}} \ge \pran{\p{|S| \ge \frac{3A}{\theta(z)^2}}-o_n(1)}\pran{1-\frac{3C}{A} - 4\theta(z)}\, .
\end{equation}
To prove such a lower bound, we shall use the second moment method; for this we require an upper bound on expectations of the form $\Ehat{N_j N_l}$, for $g < j \le l \le d$, and we now turn to proving such a bound.

The principal contribution comes from the case $j=l$; in this case we seek an upper bound on 
\begin{align*}
\Ehat{N_\ell^2} 	& = \sum_{u,u'\not\in V(M_{n}(d))} \phat{u \in N_\ell,u' \in N_\ell} 
\end{align*}
We bound the last probability above by considering whether or not $u'$ lies within $\cT(u)$. 
If $u \in N_\ell$ then $|\cT(u)| \le C/\gamma^2$; by the symmetry of the set $\{w: w \not\in V(M_n(d))\}$, we thus have 
\begin{align}
\phat{u \in N_\ell,u' \in V(\cT(u))} & = \phat{u' \in V(\cT(u))|u \in N_\ell}\cdot \phat{u \in N_\ell}  \nonumber\\
& \le \frac{C/\gamma^2}{(1-\eps)\theta(z)n}\cdot \phat{u \in N_\ell}\, \nonumber\\
& \le \I{E_{n,z}}\frac{C}{(\gamma(1-\eps)\theta(z)n)^2} .\label{eq:twoinonetree}
\end{align}
We next bound $\phat{u \in N_\ell,u' \in N_\ell, u' \not\in V(\cT(u))}$. If $u \in N_\ell$ and $u'\not\in V(\cT(u))$, then in order to have $u' \in N_\ell$, the $z$-{\it Prim}$(u')$ procedure must at some point add a vertex of $\cT(u)$. To bound the latter probability, consider a modification of $z$-{\it Prim}$(u')$ which stops the first time {\em either} a vertex of $M_n(d)$ {\em or} a vertex of $\cT(u)$ is added. Write $\beta(u')$ for the last vertex added by the modified procedure; then 
\[
\phat{u \in N_\ell,u' \in N_\ell, u' \not\in V(\cT(u))} = \phat{u \in N_\ell, u' \not\in V(\cT(u)),\beta(u') \in V(\cT(u))}\, .
\]
We have already conditioned on $M_{n}(d)$ (more precisely, on $\cF_{n}(d,z)$); we now additionally 
condition on $\cT(u)$. All edges from $\cT(u)$ to the rest of the graph have weight at least $z$, and it follows by arguing as at (\ref{stochasticordering}) that 
\begin{align*}
& \quad \phat{\left.u' \not\in V(\cT(u)),\beta(u') \in V(\cT(u)),u \in N_\ell~\right|~\cT(u)} \\
\le & \quad \I{u' \not\in V(\cT(u))} \I{u \in N_\ell}\I{E_{n,z}} \cdot \frac{|\cT(u)|}{d-g+|\cT(u)|}\, .
\end{align*}
When $u \in N_\ell$ we have $|\cT(u)| \le C/\gamma^2$, and on $E_{n,z}$ we have $d \ge (1-\eps)\theta(z)\cdot n$. It follows that 
\[
\phat{u' \not\in V(\cT(u)),\beta(u') \in V(\cT(u)),u \in N_\ell~|~\cT(u)}
\le 
\I{u\in N_\ell}\cdot \I{E_{n,z}}\cdot\frac{C/\gamma^2}{(1-\eps)\theta(z)\cdot n}\, .
\]
The only term on the right that depends on $\cT(u)$ is $\I{u \in N_\ell}$; by averaging over $\cT(u)$ we thus obtain
\begin{align}
\phat{u \in N_\ell,u' \in N_\ell, u' \not\in V(\cT(u))} & \le \phat{u \in N_\ell} \cdot \frac{C/\gamma}{\gamma^2(1-\eps)\theta(z)n}\, \\
& \le \I{E_{n,z}}\frac{C}{(\gamma(1-\eps)\theta(z)n)^2}\, .
\end{align}
Combined with (\ref{eq:twoinonetree}) we thus obtain the bound 
\[
\phat{u \in N_\ell,u' \in N_\ell} \le \I{E_{n,z}}\cdot \frac{2C}{(\gamma(1-\eps)\theta(z)n)^2}\, ,
\]
and summing over pairs $u,u' \not\in V(M_n(d))$ (there are less than $n^2$ such pairs) yields
\begin{equation}\label{eq:nksquarebd}
\Ehat{N_\ell^2} \le \I{E_{n,z}} \cdot \frac{2C}{(\gamma(1-\eps)\theta(z))^2}
\end{equation}
In the case $j \ne \ell$, the same style of argument works with minor modifications, which we only briefly sketch. 
In order to have $u \in N_j$ and $u' \in N_\ell$ we can not have $u' \in V(\cT(u))$, and must have 
$\alpha(u')=v_\ell$. The symmetry of the nodes $\{v_i, g < i \le d\}$ is not broken by the knowledge that $u \in N_j$, so we have 
\[
\phat{\alpha(u')=v_\ell~|~u \in N_j} \le \I{E_{n,z}} \cdot \frac{1}{(1-\eps)\theta(z)n)}\, ;
\] 
we thus obtain the bound 
\[
\phat{u \in N_j,u' \in N_\ell} \le \I{E_{n,z}} \frac{1}{((1-\eps) \theta(z) n)^2}\, ,
\]
and so 
\[
\Ehat{N_jN_\ell} \le \I{E_{n,z}} \frac{1}{((1-\eps)\theta(z))^2}\, .
\]
Combining the preceding inequality with (\ref{eq:nksquarebd}), it follows that for any $\cF_{n}(d,z)$-measurable set $S \subset \{v_i, g < i \le d\}$, writing $N_S = \sum_{\{i: v_i \in S\}} N_i$, we have 
\[
\Ehat{N_S^2} 
\le \I{E_{n,z}} \frac{(|S|^2 + 2C|S|/\gamma^2)}{(1-\eps)^2\theta(z)^2}\, .
\]
By (\ref{eq:nklower}), we also have 
\[
\Ehat{N_S} 
\ge \I{E_{n,z}}\frac{(1-2\theta(z)))|S|}{(1+\eps)\theta(z)}.
\]
By the preceding inequalities and the conditional Chebyshev inequality (\cite{durrett10prob}, p.194), for any $t > 0$ we have 
\begin{align*}
\phat{|N_S-\Ehat{N_S}| > t} & \le \frac{1}{t^2} \pran{\Ehat{N_S^2}-\Ehat{N_S}^2} \\
					& \le  \frac{1}{t^2} \frac{|S|^2}{\theta(z)^2}
					\pran{\frac{1+2C/(\gamma^2|S|)}{(1-\eps)^2} - \frac{(1-2\theta(z))^2}{(1+\eps)^2}}
\end{align*}
Recall that for all $z > 1$ we have $\theta(z) \le 2(z-1) = 2\gamma$. 
It follows that for $z-1$ sufficiently small, for any $D > 1$, for $n$ large enough we have 
\begin{align*}
\phat{N_S \ge \frac{|S|}{3\theta(z)}, |S| \ge \frac{DC}{\theta(z)^2}}
& \ge \phat{|N_S-\Ehat{N_S}| \le \frac{|S|}{2\theta(z)}, |S| \ge \frac{DC}{\theta(z)^2}} \\
& \ge \I{|S| \ge DC/\theta(z)^2}\cdot \I{E_{n,z}}\cdot \pran{1-\pran{\frac{1+8/D}{(1-\eps)^2} - \frac{(1-2\theta(z))^2}{(1-\eps)^2}}}.
\end{align*}
By the tower law, we thus obtain 
\begin{align*}
\p{N_S \ge \frac{DC}{3\theta(z)^3}} & \ge \p{N_S \ge \frac{|S|}{3\theta(z)}, |S| \ge \frac{DC}{\theta(z)^2},E_{n,z}} \\
			& = \E{\phat{N_S \ge \frac{|S|}{3\theta(z)}, |S| \ge \frac{DC}{\theta(z)^2}}} \\
			& \ge \p{|S| \ge DC/\theta(z)^2,E_{n,z}} \cdot \pran{1-\pran{\frac{1+8/D}{(1-\eps)^2} - \frac{(1-2\theta(z))^2}{(1-\eps)^2}}}\, .
\end{align*}
Since $\p{E_{n,z}} \to 1$ as $n \to \infty$, and since $\eps > 0$ was arbitrary, it follows that for any $\cF_{n,d}$-measurable $S \subset \{v_j, g < j \le d\}$, 
\[
\p{N_S \ge \frac{DC}{3\theta(z)^3}} \ge \pran{\p{|S|\ge \frac{DM}{\theta(z)^2}} - o_n(1)} \cdot \pran{1-\frac{8}{D} - 4\theta(z)}\, .
\]
Taking $A=D/3$ then establishes (\ref{eq:nstoprove}) and so completes the proof of Proposition~\ref{prop:lowerboundkey}. 

\printnomenclature[2cm]

\end{document}